\documentclass[onefignum,onetabnum]{siamart220329}

\usepackage{amsmath,amssymb} %
\usepackage{amsfonts}
\usepackage{graphicx}
\usepackage{booktabs}
\usepackage{siunitx}

\usepackage{tikz}

\newsiamremark{remark}{Remark}
\newcommand{\doi}[1]{DOI\@: \href{http://dx.doi.org/#1}{\nolinkurl{#1}}}

\newcommand{\overbar}[1]{\mkern 1.5mu\overline{\mkern-1.5mu#1\mkern-1.5mu}\mkern 1.5mu}
\newcommand{\Nnode}{N_{\text{node}}}
\newcommand{\Ncluster}{N_{\text{cluster}}}
\newcommand{\Nlevel}{N_{\text{level}}}
\DeclareMathOperator*{\argmax}{argmax}

\newtheorem{assumption}{Assumption}[section]

\headers{Generalizing Lloyd's algorithm for graph clustering}{Zaman,
  Nytko, Taghibakhshi, MacLachlan, Olson, West}

\title{Generalizing Lloyd's algorithm for graph clustering\thanks{Submitted to the editors TODO.\funding{The work of S.P.M. was partially supported by an NSERC Discovery Grant.
This material is based in part upon work supported by the Department of Energy,
National Nuclear Security Administration, under Award Number \textit{DE-NA0003963}.}}}

\author{
Tareq Zaman\thanks{Interdisciplinary Program in Scientific
  Computing, Memorial University of Newfoundland, St.\ John's, NL,
  Canada (\email{tzaman@mun.ca}, \email{smaclachlan@mun.ca}).}
\and Nicolas Nytko\thanks{Department of Computer Science,
University of Illinois Urbana-Champaign, Urbana, IL 61801, USA
(\email{nnytko2@illinois.edu}, \email{lukeo@illinois.edu}).}
\and Ali Taghibakhshi\thanks{Department of Mechanical Science and
  Engineering, University of Illinois Urbana-Champaign, Urbana, IL
  61801, USA (\email{alit2@illinois.edu},
  \email{mwest@illinois.edu}).}
\and Scott MacLachlan\footnotemark[2]
\and Luke Olson\footnotemark[3]
\and Matthew West\footnotemark[4]}

\ifpdf%
\hypersetup{%
  pdftitle={Generalizing Lloyd's algorithm for graph clustering},
  pdfauthor={Zaman, Nytko, Taghibakhshi, MacLachlan, Olson, West}
}
\fi

\definecolor{tab-blue}{HTML}{1f77b4}
\definecolor{tab-gray}{HTML}{7f7f7f}

\begin{document}

\maketitle

\begin{abstract}
  Clustering is a commonplace problem in many areas of data science,
  with applications in biology and bioinformatics, understanding
  chemical structure, image segmentation, building recommender systems, and
  many more fields.  While there are many different clustering
  variants (based on given distance or graph structure, probability
  distributions, or data density), we consider here the problem of
  clustering nodes in a graph, motivated by the problem of aggregating
  discrete degrees of freedom in
  multigrid and domain decomposition methods for solving sparse linear systems.
  Specifically, we consider the challenge of forming \textit{balanced} clusters
  in the graph of a sparse matrix for use in algebraic multigrid, although
  the algorithm has general applicability.  Based on an extension of the Bellman-Ford algorithm, we generalize Lloyd's algorithm for partitioning subsets of $\mathbb{R}^n$
  to balance the number of nodes in each cluster; this is
  accompanied by a rebalancing algorithm that reduces the overall \textit{energy} in the system.
  The algorithm provides control over the number of clusters and leads to ``well centered''
  partitions of the graph.  Theoretical results are provided to establish linear complexity
  and numerical results in the context of algebraic multigrid highlight the benefits of improved
  clustering.
\end{abstract}



\section{Introduction}

Consider a directed graph $G(V,E,W)$ where $V$ is a set of nodes (or vertices),
$V=\{1,\dots,\Nnode\}$, and where $E$ is a list of edges given by $E = \{(i,j) \mid
W_{i,j}\ne 0\}$ for some weight matrix $W$. The (sparse) weight matrix $W$ is
assumed to have non-negative off-diagonal entries and zero diagonal entries.
The goal of this work is to take an initial set of clusters and to improve them
via local operations which reduce a quadratic energy functional. These
operations should have linear complexity in the number of nodes for
appropriately sparse graphs and initial seedings
(see~\cref{ass:graph_complexity}).

There are an array of challenges in clustering; the focus here is twofold: (1)
developing efficient algorithms where the number of clusters $\Ncluster$ can be
specified; and (2) generating clusterings that are considered ``well
balanced''.  As a motivating example, we consider a graph generated from a
finite-element discretization on a unit disk with 528 vertices\footnote{This
example is studied in detail in~\cref{sec:numerics}}.
\Cref{fig:disc_allmethods} illustrates the clustering of nodes
using four different methods that underscore these two challenges.
Nearest-neighbor (or \textit{Greedy}) clustering~\cite{vanaek1996algebraic}
yields 63 clusters in
this case.  While this simple algorithm lacks control of the number of
clusters, the clustering offers a clear balance in the number of nodes per
cluster and total diameter of each cluster.  In contrast, with $\Ncluster=52$
the spectral-based partitioner METIS~\cite{karypis1997metis} yields long
clusters (and large diameters).
\begin{figure}
  \centering
  \includegraphics{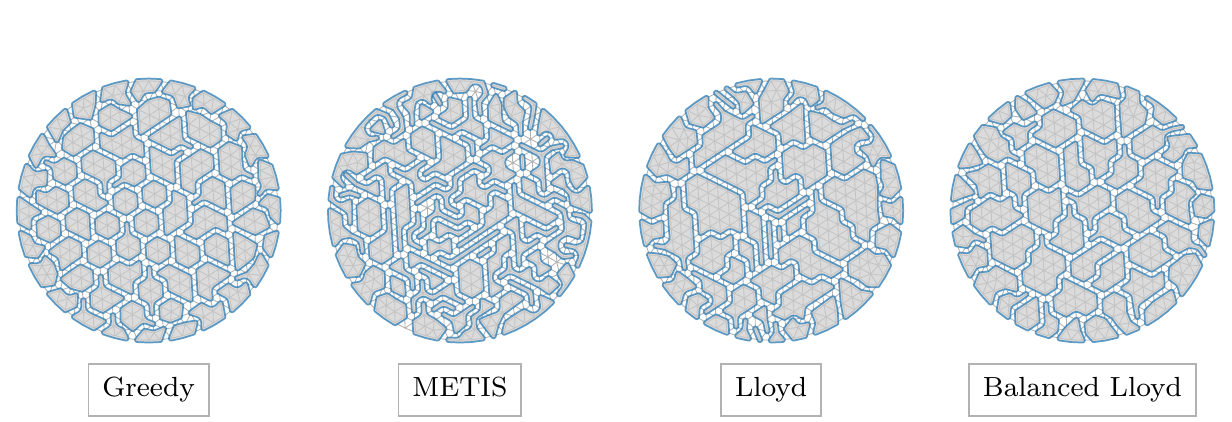}
  \caption{Example clusterings.}\label{fig:disc_allmethods}
\end{figure}

In this work, we focus on shortest-path
based clustering algorithms. \textit{Lloyd} clustering (also known as Lloyd aggregation)~\cite{bell2008algebraic}, for
example, uses Bellman-Ford~\cite{erickson2019} to construct $\Ncluster=52$
clusters based on an initial seeding.  Overall clustering quality depends
\textit{highly} on the initial seeding; often even costly $\mathcal{O}(\Nnode^3)$
algorithms for seeding, such as $k$-means++~\cite{kmeanspp}, do not
dramatically improve the final clustering in this case.  Lastly, we highlight
clustering based on an algorithm introduced here: a balanced form of Lloyd
clustering with \textit{rebalancing} to minimize diameter.  While standard
Lloyd clustering results in both large and small clusters, the cluster shapes
with \textit{balanced} Lloyd clustering and \textit{rebalancing} are more
consistent.

While there is a long history of aggregation-based multigrid methods (cf.~\cite{mika1992acceleration, vanvek1992acceleration, vanaek1996algebraic,notay2010aggregation}), surprisingly little attention has been paid to the influence of cluster quality on the performance of the resulting algorithm.  The \textit{greedy} clustering algorithm originally proposed in~\cite{mika1992acceleration} has become a standard approach that is used in many codes.  Some variants on this approach have been introduced for massively parallel settings; most notably, approaches based on distance-two maximal independent sets in the graph~\cite{RSTuminaro_CTong_2000a, bell2012exposing, KelleyRajamanickam2022}.  Both of these approaches make minimal use of the weight matrix, $W$, aside from using its nonzero pattern to infer binary connectivity data in the graph.  In contrast, in~\cite{bell2008algebraic}, Lloyd's algorithm~\cite{lloyd1982least} was extended from computing Voronoi diagrams in $\mathbb{R}^n$ to computing clusters in graphs, using the values in $W$ to define graph distances.  It is this approach that we extend here.

In this paper, we introduce a general clustering method for use in graph
partitioning and algebraic multigrid that provides control of the number of
clusters, yields ``centered'' clusters, and can be implemented with
off-the-shelf codes for Bellman-Ford and Floyd-Warshall
algorithms~\cite{erickson2019}.
All algorithms are implemented in and are available through the open source
package PyAMG~\cite{BeOlSc2022}.
In~\cref{sec:AMG}, we review
aggregation-based algebraic multigrid (AMG) and survey the Lloyd clustering algorithm.
\Cref{sec:balancedlloyd} introduces \textit{balanced} Lloyd clustering and a
\textit{rebalance} algorithm, along with theoretical evidence of convergence
and complexity.  Finally, \cref{sec:numerics} provides numerical evidence in
support, expanding the example in~\cref{fig:disc_allmethods} and
others.

\textit{Note: throughout the paper and embedded in the algorithms, we make use
of the notation listed in~\cref{tab:symbols}.}

\section{Clustering in algebraic multigrid}\label{sec:AMG}

Algebraic multigrid methods are a family of iterative methods for the solution of sparse linear systems of the form $Au=f$, where $A$ is an $\Nnode\times\Nnode$ matrix and $u$ and $f$ are vectors of dimension $\Nnode$.  Like all multigrid methods, they achieve their efficiency through the use of two complementary processes, known as relaxation and coarse-grid correction.  For algebraic multigrid methods, we typically consider a fixed relaxation scheme (such as a stationary weighted Jacobi or Gauss-Seidel iteration on the linear system) and seek to compute a coarse-grid correction process that adequately complements relaxation to lead to an efficient solution algorithm.  In aggregation-based methods, the coarse-grid correction process takes the form of first computing a clustering of the fine-grid degrees of freedom (nodes in the graph of the sparse matrix, $A$), and then computing an interpolation operator from the clustered degrees of freedom to those on the fine grid.
Rootnode-based aggregation methods additionally make use of a center that is identified for each cluster~\cite{RootNode}.
For a more thorough review of algebraic multigrid methods, see~\cref{app:AMG} or~\cite{WLBriggs_VEHenson_SFMcCormick_2000a, UTrottenberg_etal_2001a}.

\Cref{fig:disc_allmethods} illustrates the wide range of
clusters that can arise for a single problem.  We next detail three common
approaches to clustering (used in the context of AMG), before introducing
a balanced method in the next section.
First, however, we define a clustering or aggregation of $G(V,E,W)$,
as in \cref{def:clustering}.  We note that the clustering is a non-overlapping covering.
\begin{definition}\label{def:clustering}
  A \emph{clustering} or \emph{aggregation} of the connected graph $G(V,E,W)$
  is a pair $(m, c)$, where $m_i$ is the cluster membership of vertex $i$
  and $c_a$ is the global index of the center for cluster $a$.  Then $m$ and
  $c$ have the following properties:
  \begin{enumerate}
    \item For each $i \in \{1,\dots,\Nnode\}$, there exists a unique $a$ with $1\leq a \leq \Ncluster$ such that $m_i = a$;
    \item For each $a\in\{1,\dots,\Ncluster\}$, for every $(i,j)$ with $m_i=m_j=a$, there exists
      a sequence $k_1, \dots, k_p$ where $m_k=a$ for $k \in \{k_1,\dots, k_p\}$ and with
      $(i,k_1),(k_q,k_{q+1}),(k_p,j)\in E$ for $q \in \{1,\dots,p-1\}$; and
    \item For each $a\in\{1,\dots,\Ncluster\}$, we have $1\leq c_a \leq \Nnode$ and $m_{c_a} = a$.
  \end{enumerate}
  The first point ensures that the clustering is a non-overlapping covering, the second requires that the subgraph over the cluster remains connected, and the third confirms that an element of each cluster is identified as the center for that cluster.
\end{definition}

\subsection{Standard AMG clustering}\label{sec:amg-agg}

\Cref{sec:standard-agg} presents two standard clustering strategies used in algebraic multigrid, the greedy algorithm in~\cref{sec:greedy-agg} and the maximal independent set algorithm in~\cref{sec:mis}.  Both of these algorithms share the same shortcoming: they do not allow the user to specify the number of clusters that are returned. This important property is determined solely based on the connectivity of the matrix and the ordering of the sequential passes they make through the matrix.  Nonetheless, these have been observed to be very effective for a certain range of problems, particularly those that arise from low-order finite-element discretization on (nearly) structured meshes in two and three spatial dimensions.

\subsection{Standard Lloyd clustering}\label{sec:lloyd}

In contrast to the standard clustering strategies, Lloyd clustering,
introduced in~\cite{bell2008algebraic}, finds a specified number of clusters ($N_{cluster}$) based on an initial seeding of the clusters.  Lloyd clustering can be viewed as an extension of
Lloyd's algorithm~\cite{lloyd1982least} applied to graphs, where an initial
random seeding of centers yields Voronoi cells (or a set of nodes closest to
each center), followed by a recentering of center locations.

A full algorithm is given in~\cref{alg:lloyd_agg}, where a subset of
$\Ncluster$ nodes are randomly selected as the initial centers, input as $c$. A standard
Bellman-Ford algorithm (see~\cref{alg:bell_ford} and~\cite[Section~8.7]{erickson2019}) is used to find the distance
and index of the closest center; the set of  points closest to each center form
the initial clustering. Next, the border nodes of each cluster are selected and a modified form of the
Bellman-Ford algorithm then identifies the (new)
center~---~see~\cref{alg:interior-nodes}~---~by selecting the node of maximum
distance to the cluster boundary (with ties selected arbitrarily). The steps
are repeated until the algorithm has converged or a maximum number of iterations (given as $T_{\text{max}}$) is reached.
\begin{algorithm}[!ht]
\caption{Lloyd clustering algorithm. See \cref{tab:symbols} for variable definitions.}\label{alg:lloyd_agg}
\begin{algorithmic}[1]
  \Function{lloyd-clustering}{$W,c, T_{\text{max}}$}
  \State $t=0$
\Repeat
  \State $m, d \gets \textsc{bellman-ford}(W, c)$\Comment{find closest centers}
  \State $c \gets \textsc{most-interior-nodes}(W, m)$\Comment{recenter}
  \State $t = t+1$
\Until{$t=T_{\text{max}}$ or no change in $c$ and $m$}
\State \textbf{return} $m, c$
\EndFunction
\end{algorithmic}
\end{algorithm}
\begin{algorithm}[!ht]
\caption{Bellman-Ford algorithm to compute distance and index of closest center. See \cref{tab:symbols} for variable definitions.}\label{alg:bell_ford}
\begin{algorithmic}[1]
\Function{bellman-ford}{$W,c$}
\State $d_i \gets \infty$ for all $i = 1,\ldots,\Nnode$\Comment{initial distance}
\State $m_i \gets 0$ for all $i = 1,\ldots,\Nnode$\Comment{initial membership undefined}
\For {$a \gets 1,\ldots,\Ncluster$}
   \State $i \gets c_a$ \Comment{cluster $a$ has center node $i$}
   \State $d_i \gets 0$ \Comment{distance of a center node to itself is zero}
   \State $m_i \gets a$ \Comment{center node $i$ belongs to its own cluster}
\EndFor
\Repeat \label{line:bell_ford_outer_loop}
    \State $\text{done} \gets \text{true}$
    \For {$i,j$ such that $W_{i,j} > 0$} \Comment{all pairs of adjacent nodes}
        \If {$d_i + W_{i,j} < d_j$} \Comment{node $j$ is closer to $i$'s center}\label{line:bell_ford_cond}
             \State $m_j \gets m_i$ \Comment{switch node $j$ to the same cluster as $i$}\label{line:bell_ford_update_m}
             \State $d_j \gets d_i + W_{i,j}$ \Comment{use the shorter distance via node $i$}
             \State $\text{done} \gets \text{false}$ \Comment{change was made; do not terminate}
        \EndIf
    \EndFor
\Until{done}
\State \textbf{return} $m,d$
\EndFunction
\end{algorithmic}
\end{algorithm}
\begin{algorithm}[!ht]
\caption{Find the most-interior node (furthest from boundary) for each cluster. See \cref{tab:symbols} for variable definitions.}\label{alg:interior-nodes}
\begin{algorithmic}[1]
\Function{most-interior-nodes}{$W,m$}
    \State $B \gets \{\}$ \Comment{border nodes}
    \For{$i,j $ such that $W_{i,j} > 0$} \Comment{all pairs of adjacent nodes}
        \If{$m_i \neq m_j$} \Comment{are nodes $i$ and $j$ in different clusters?}
            \State $B \gets B \cup \{i, j\}$ \Comment{if so, add both of them to the border set}
        \EndIf
    \EndFor
    \State $\cdot, d \gets \textsc{bellman-ford}(W, B)$ \Comment{$d$ is distance from cluster borders}
    \For{$i \gets 1,\ldots,\Nnode$}
        \State $a \gets m_i$ \Comment{$a$ is the cluster index for node $i$}
        \State $c_a \gets i$ \Comment{assign the highest-index node as cluster center}
    \EndFor
    \For{$i \gets 1,\ldots,\Nnode$}
        \State $a \gets m_i$ \Comment{$a$ is the cluster index for node $i$}
        \State $j \gets c_a$ \Comment{$j$ is the current cluster center}
        \If{$d_i > d_j$} \Comment{is node $i$ further from the border than $j$?}
            \State $c_a \gets i$ \Comment{if so, node $i$ is the new cluster center}
        \EndIf
    \EndFor
    \State \textbf{return} $c$
\EndFunction
\end{algorithmic}
\end{algorithm}

\subsubsection{Theoretical observations}

A significant advantage of standard Lloyd clustering, as
in~\cref{alg:lloyd_agg}, is the dependence on \textit{off-the-shelf} algorithms
such as Bellman-Ford.  This allows us to establish key properties that will
carry over to more advanced algorithms in the next section.

To begin, we note that standard Bellman-Ford terminates
(in~\cref{thm:bfterminates}), an important property to maintain as we seek
more balanced clusters.
\begin{theorem}\label{thm:bfterminates}
  \Cref{alg:bell_ford} terminates.
\end{theorem}
\begin{proof}
  This is a standard result~\cite[Section~8.7]{erickson2019}.
\end{proof}
Likewise, while we assume the initial graph is connected, \cref{def:clustering}
requires each of the clusters to be connected.  Bellman-Ford provides this, as summarized
in~\cref{thm:bell_ford_conn}.
\begin{theorem}\label{thm:bell_ford_conn}
  The clusters returned by \cref{alg:bell_ford} are connected.
\end{theorem}
\begin{proof}
  This follows from the proof of \cref{thm:bal_bell_ford_conn}, using only the first case in the proof corresponding to \cref{line:bal_bell_ford_cond1} in \cref{alg:bal_bell_ford}.
\end{proof}

\section{Balanced Lloyd clustering}\label{sec:balancedlloyd}

Lloyd clustering in~\cref{sec:lloyd} enables the construction of a
\textit{variable} number of clusters, based on the initial seeding.  Yet, the
method can result in poor \textit{quality}
clusters~(cf.~\cref{fig:disc_allmethods}).  As an example,
consider a nearest-neighbor weight matrix $W$ based on distance and on a $6\times
6$ structured mesh.  \Cref{fig:bad_aggs} illustrates two common scenarios in
standard Lloyd clustering.  The first is the emergence of long, narrow
clusters. This is, in part, due to the method of finding boundaries in
\textsc{most-interior-nodes}; in this case, the entire cluster (left figure) is comprised of boundary
nodes, leaving no opportunity to re-center.  The second artifact of standard
Lloyd is that of disparate cluster sizes.  Here, we observe both
large clusters and clusters of a single point (right figure).  An immediate goal in the algorithms
of this section is to address these two points.
\begin{figure}
  \centering
  \includegraphics{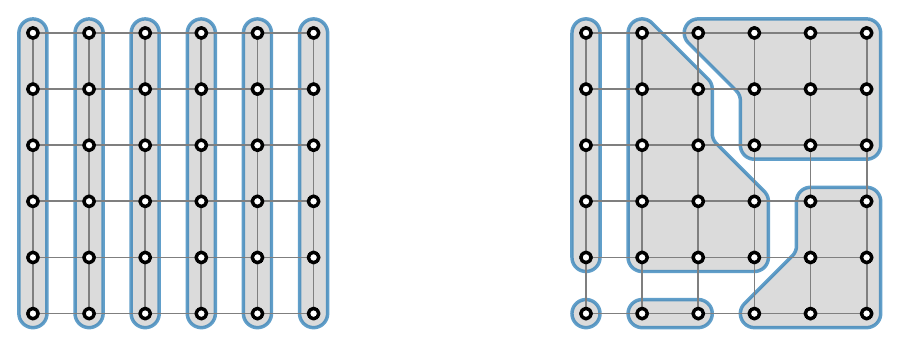}
  \caption{Two example clusterings from Lloyd clustering on a $6 \times 6$ mesh.}\label{fig:bad_aggs}
\end{figure}

The Bellman-Ford algorithm is the central component of standard Lloyd, finding
the shortest path for each seed and being used to identify the center by the most distal
points to each boundary node in a cluster.
In the following, we make use of the total squared distance or the total energy, defined by $k$-means algorithms~\cite{clrs_algorithms},
as the sum of the squares of the distances of any node to its center:
\begin{equation}\label{eq:energy}
  H = \sum_{i=1}^{\Nnode} d_i^2,
\end{equation}
where $d_i$ is defined to be the distance from node $i$ to the center of the
cluster for node $i$, namely $c_a$ where cluster $a$ satisfies $a = m_i$.
In the following, we detail a method for centering
nodes (see~\cref{alg:center_nodes}), where we seek to minimize the energy
in the cluster.  Using the energy, as defined in~\cref{eq:energy},
requires the computation of the shortest path for each pair of nodes in the
cluster, a.k.a.\ the all-pairs shortest path problem.  For this we turn to a
per-cluster use of Floyd-Warshall~\cite{erickson2019} as detailed
in~\cref{alg:floyd_warshall}.
\begin{algorithm}[!ht]
\caption{Floyd-Warshall algorithm~\cite[Section~9.8]{erickson2019} to find inter-node distances within each cluster. See \cref{tab:symbols} for variable definitions.}\label{alg:floyd_warshall}
\begin{algorithmic}[1]
\Function{clustered-floyd-warshall}{$W,m$}
    \State $V_a \gets \{i \mid m_i = a\}$ for all $a = 1,\ldots,N_{\rm cluster}$ \Comment{nodes in cluster $a$}

    \For{$a \gets 1,\ldots,N_{\rm cluster}$}

    \For{$i,j \in V_a$}
    \State $D_{i,j} \gets \infty$ \Comment{initial distance $i \to j$}
    \State $P_{i,j} \gets 0$ \Comment{initial predecessor node for $i \to j$}
    \If{$W_{i,j} > 0$}
    \State $D_{i,j} \gets W_{i,j}$ \Comment{adjacent nodes have the adjacency distance}
    \State $P_{i,j} \gets i$ \Comment{the predecessor is the tail node for adjacent pairs}
    \EndIf
    \If{$i = j$}
    \State $D_{i,i} \gets 0$ \Comment{nodes are distance zero from themselves}
    \State $P_{i,i} \gets i$ \Comment{nodes are their own predecessors to themselves}
    \EndIf
    \EndFor %

    \For{$k \in V_a$} \Comment{potential intermediate node on the path $i \to j$}
    \For{$i,j \in V_a$} \Comment{all other node pairs within the cluster}
    \If{$D_{i,k} + D_{k,j} < D_{i,j}$} \Comment{$i \to k \to j$ shorter than $i \to j$}\label{line:epsfw}
    \State $D_{i,j} \gets D_{i,k} + D_{k,j}$ \Comment{switch to the shorter distance}
    \State $P_{i,j} \gets P_{k,j}$ \Comment{take the predecessor from $k \to j$}
    \EndIf
    \EndFor %
    \EndFor %

    \EndFor %
    \State \textbf{return} $D,P$
\EndFunction
\end{algorithmic}
\end{algorithm}

For cluster $a$, we note that the calculation of shortest paths
in~\cref{alg:floyd_warshall} is $\mathcal{O}(s_a^3)$, where $s_a = |\{i \mid
m_i = a\}|$ is the size of cluster $a$.  In the following section, we
establish linear complexity in the number of nodes, $\Nnode$, with assumptions
on the maximum cluster size.

The use of energy as a target (see~\cref{eq:energy}) provides an
opportunity to rebalance the clustering to account for small or large
clusters.  For this, we introduce a rebalancing algorithm that calculates the energy
increase in splitting clusters and the energy decrease in eliminating clusters.
The overall process relies on the distances from Floyd-Warshall.
In~\cref{sec:balancedlloydalg}, a balanced version of Bellman-Ford is
introduced, leading to a balanced form of Lloyd clustering.  The rebalancing
algorithm is constructed in~\cref{sec:rebalance} followed by theoretical
observations.  The rebalanced Lloyd algorithm requires several components
and we summarize the dependence in~\cref{fig:algdependence}.
\begin{figure}[!ht]
  \centering
\tikzset{every picture/.style={>=latex,
                               line width=0.5pt,
                               line cap=round,
                           }}
\newcommand{\myref}[1]{\tiny\labelcref{#1}}
\newcommand{\myspace}{\hspace*{0.4cm}}
\newcommand{\mysep}{1pt}
\newcommand{\myosep}{2pt}
\begin{tikzpicture}[x=25pt, y=20pt]
  \begin{scope}
    \node[draw=black!20] (lloydagg)     at (0,  0) [anchor=east] {\textsc{lloyd-clustering}\myspace};
    \node[draw=black!10, anchor=east, inner sep=\mysep, outer sep=\myosep] at (lloydagg.east) {\myref{alg:lloyd_agg}};
    \node[draw=black!20] (bellmanford)  at (1,  0) [anchor=west] {\textsc{bellman-ford}\myspace};
    \node[draw=black!10, anchor=east, inner sep=\mysep, outer sep=\myosep] at (bellmanford.east) {\myref{alg:bell_ford}};
    \node[draw=black!20] (mostinterior) at (1, -1) [anchor=west] {\textsc{most-interior}\myspace};
    \node[draw=black!10, anchor=east, inner sep=\mysep, outer sep=\myosep] at (mostinterior.east) {\myref{alg:interior-nodes}};

    \draw[->, line width=1pt] (lloydagg.east) -- (bellmanford.west);
    \draw[->, line width=1pt] (lloydagg.south east) -- (mostinterior.north west);
    \draw[->, line width=1pt] (mostinterior.north) -- (bellmanford.south);
  \end{scope}
  \begin{scope}[shift={(0pt, -70pt)}]
    \node[draw=black!20] (blloydagg)     at (0,  0)  [anchor=east] {\textsc{balanced-lloyd-clustering}\myspace};
    \node[draw=black!10, anchor=east, inner sep=\mysep, outer sep=\myosep] at (blloydagg.east) {\myref{alg:bal_lloyd_agg}};
    \node[draw=black!20] (binit)         at (1,  0)  [anchor=west] {\textsc{balanced-initialization}\myspace};
    \node[draw=black!10, anchor=east, inner sep=\mysep, outer sep=\myosep] at (binit.east) {\myref{alg:bal_lloyd_init}};
    \node[draw=black!20] (bbellmanford)  at (1, -1)  [anchor=west] {\textsc{balanced-bellman-ford}\myspace};
    \node[draw=black!10, anchor=east, inner sep=\mysep, outer sep=\myosep] at (bbellmanford.east) {\myref{alg:bal_bell_ford}};
    \node[draw=black!20] (centernodes)   at (1, -2)  [anchor=west] {\textsc{center-nodes}\myspace};
    \node[draw=black!10, anchor=east, inner sep=\mysep, outer sep=\myosep] at (centernodes.east) {\myref{alg:center_nodes}};
    \node[draw=black!20] (floydwarshall) at (1, -3)  [anchor=west] {\textsc{clustered-floyd-warshall}\myspace};
    \node[draw=black!10, anchor=east, inner sep=\mysep, outer sep=\myosep] at (floydwarshall.east) {\myref{alg:floyd_warshall}};

    \node[draw=black!20!red] (rblloydagg)    at (0, -4)    [anchor=east] {\textsc{rebalanced-lloyd-clustering}\myspace};
    \node[draw=black!10, anchor=east, inner sep=\mysep, outer sep=\myosep] at (rblloydagg.east) {\myref{alg:rebal_lloyd_agg}};
    \node[draw=black!20!red] (rebalance)     at (1, -4)    [anchor=west] {\textsc{rebalance}\myspace};
    \node[draw=black!10, anchor=east, inner sep=\mysep, outer sep=\myosep] at (rebalance.east) {\myref{alg:rebalance}};
    \node[draw=black!20!red] (elimpenalty)   at (2.7, -5)  [anchor=west] {\textsc{elimination-penalty}\myspace};
    \node[draw=black!10, anchor=east, inner sep=\mysep, outer sep=\myosep] at (elimpenalty.east) {\myref{alg:elimination_penalty}};
    \node[draw=black!20!red] (splitimprove)  at (2.7, -6)  [anchor=west] {\textsc{split-improvement}\myspace};
    \node[draw=black!10, anchor=east, inner sep=\mysep, outer sep=\myosep] at (splitimprove.east) {\myref{alg:split_improvement}};
    \node[draw=black!20!red] (markunavail)   at (2.7, -7)  [anchor=west] {\textsc{mark-unavailable}\hspace*{0.2cm}\myspace};
    \node[draw=black!10, anchor=east, inner sep=\mysep, outer sep=\myosep] at (markunavail.east) {\myref{alg:mark_unavailable}};

    \draw[->, line width=1pt, draw=black] (blloydagg.east) -- (binit.west);
    \draw[->, line width=1pt, draw=black] (blloydagg.east) -- (bbellmanford.north west);
    \draw[->, line width=1pt, draw=black] (blloydagg.east) -- (centernodes.north west);
    \draw[->, line width=1pt, draw=black] (blloydagg.east) -- (floydwarshall.north west);
    \draw[->, line width=1pt, draw=black!20!red] (rblloydagg.north) -- (rblloydagg.north |- blloydagg.south);
    \draw[->, line width=1pt, draw=black!20!red] (rblloydagg.east) -- (rebalance.west);
    \draw[->, line width=1pt, draw=black!20!red] (rebalance.south) -- (elimpenalty.north west);
    \draw[->, line width=1pt, draw=black!20!red] (rebalance.south) -- (splitimprove.north west);
    \draw[->, line width=1pt, draw=black!20!red] (rebalance.south) -- (markunavail.north west);
  \end{scope}
\end{tikzpicture}

  \caption{Algorithm dependence with references. Rebalancing components are highlighted in red.}\label{fig:algdependence}
\end{figure}
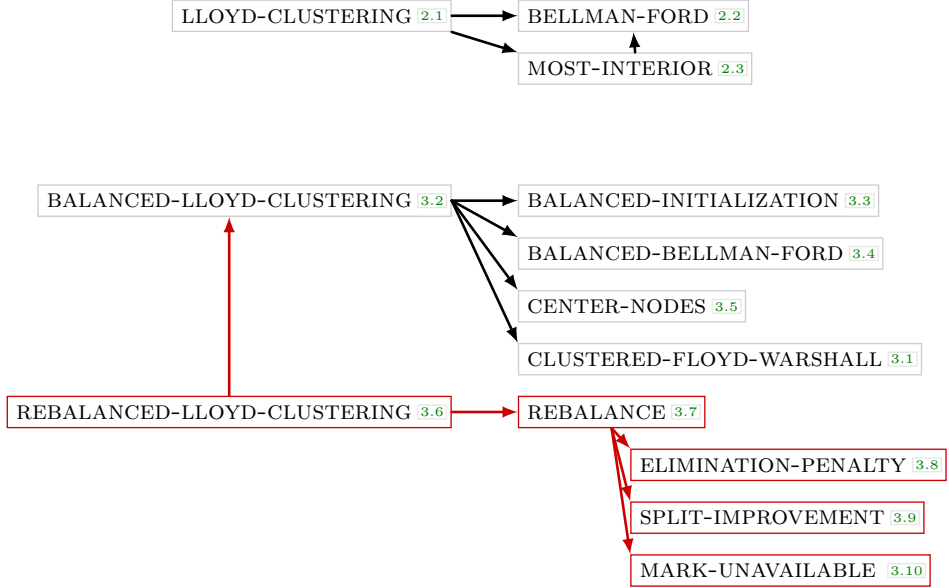

\subsection{Balanced algorithms}\label{sec:balancedlloydalg}

One disadvantage of Lloyd clustering is that the clusters are not guaranteed
(nor expected) to be uniformly sized.  In many practical settings, a node is likely to have
nearly the same distance to multiple centers.
In this case, Lloyd clustering randomly assigns the
node to a cluster; in contrast, balanced Lloyd clustering targets
uniformly sized clusters, as described
in~\cref{alg:bal_lloyd_agg}. In the balanced approach, if a
node has the same distance to different centers, the node is assigned to a
smaller cluster, leading to increased uniformity across clusters.
To find the new centroid of the cluster, we then use
Floyd-Warshall (\cref{alg:floyd_warshall})
to find the energy-minimizing node, rather than using
Bellman-Ford to find the node which is furthest from the cluster boundary.
A node of a cluster having the minimum sum of squared
distance to other cluster nodes is taken as centroid of that region
(\cref{alg:center_nodes}).  Consequently, long, narrow clusters as in~\cref{fig:bad_aggs}
will expose centers near the true center, whereas the boundary-distances used in standard
Lloyd clustering leave any boundary centers unchanged.
\begin{algorithm}[!ht]
\caption{Balanced version of Lloyd clustering. See \cref{tab:symbols} for variable definitions.}\label{alg:bal_lloyd_agg}
\begin{algorithmic}[1]
\Function{balanced-lloyd-clustering}{$W,c,T_{\text{max}}, T_{\text{BFmax}}$}
\State $m, d, p, n, s \gets \textsc{balanced-initialization}(c, \Nnode)$
\State $t=0$
\Repeat
    \State $m, d, p, n, s \gets \textsc{balanced-bellman-ford}(W, m, c, d, p, n, s, T_{\text{BFmax}})$
    \State $D,P \gets \textsc{clustered-floyd-warshall}(W, m)$
    \State $c, d, p, n \gets \textsc{center-nodes}(W, m, c, d, p, n, D, P)$
    \State $t = t+1$
\Until{$t=T_{\text{max}}$ or no change in any of $m,c,d,p,n,s$}
\State \textbf{return} $m, c, d, p, n, s, D, P$
\EndFunction
\end{algorithmic}
\end{algorithm}
\begin{algorithm}[!ht]
\caption{Initialization for balanced algorithms. See \cref{tab:symbols} for variable definitions.}\label{alg:bal_lloyd_init}
\begin{algorithmic}[1]
\Function{balanced-initialization}{$c,\Nnode$}
\State $m_i \gets 0$ for all $i = 1,\ldots,N_{\rm node}$ \Comment{cluster membership for node $i$}
\State $d_i \gets \infty$ for all $i = 1,\ldots,N_{\rm node}$ \Comment{distance to node $i$ from its cluster center}
\State $p_i \gets 0$ for all $i = 1,\ldots,N_{\rm node}$ \Comment{predecessor node for node $i$}
\State $n_i \gets 0$ for all $i = 1,\ldots,N_{\rm node}$ \Comment{number of predecessor nodes for node $i$}
\State $s_a \gets 1$ for all $a = 1,\ldots,N_{\rm cluster}$ \Comment{size of cluster $a$}
\For {$a \gets 1,\ldots,N_{\rm cluster}$}
    \State $i \gets c_a$ \Comment{$i$ is the center node index for cluster $a$}
    \State $d_i \gets 0$ \Comment{distance of center to node $i$ from itself is zero}
    \State $m_i \gets a$ \Comment{center node $i$ belongs to its own cluster}
    \State $p_i \gets i$ \Comment{centers are their own predecessors}
    \State $n_i \gets 1$ \Comment{centers have one predecessor}
\EndFor
\State \textbf{return} $m,d,p,n,s$
\EndFunction
\end{algorithmic}
\end{algorithm}
\begin{algorithm}[!ht]
  \caption{Balanced version of Bellman-Ford. See \cref{tab:symbols} for variable definitions.}\label{alg:bal_bell_ford}
\begin{algorithmic}[1]
\Function{balanced-bellman-ford}{$W,m,c,d,p,n,s, T_{\text{BFmax}}$}
\State $t \gets 0$; $z^{(t)} \gets z$ for all variables $z$ \Comment{only for use in proofs}
\Repeat
    \State $\text{done} \gets \text{true}$
    \For {$i,j$ such that $W_{i,j} > 0$} \Comment{all pairs of adjacent nodes}
    \State $\mathfrak{s}_i \gets s_{m_i} \text{ if } m_i > 0, \text{ else } 0$ \Comment{size of cluster containing node $i$}
    \State $\mathfrak{s}_j \gets s_{m_j} \text{ if } m_j > 0, \text{ else } 0$ \Comment{size of cluster containing node $j$}
        \State $\text{switch} \gets \text{false}$
        \If {$d_i + W_{i,j} < d_j$} \Comment{$j$ is closer to $i$'s center than its own}\label{line:bal_bell_ford_cond1}
            \State $\text{switch} \gets \text{true}$
        \EndIf
        \If {$d_i + W_{i,j} = d_j$} \Comment{distance to $j$ is similar from $i$'s center}\label{line:bal_bell_ford_cond2a}
        \If {$\mathfrak{s}_i + 1 < \mathfrak{s}_j$} \Comment{node $i$'s cluster is smaller (by 2 or more)}\label{line:bal_bell_ford_cond2b}
                \If {$n_j = 0$} \Comment{node $j$ is free to switch (not a predecessor)}\label{line:bal_bell_ford_cond2c}
                    \State $\text{switch} \gets \text{true}$
                \EndIf
            \EndIf
        \EndIf
        \If {switch}
            \State $s_{m_i} \gets \mathfrak{s}_i + 1$, $s_{m_j} \gets \mathfrak{s}_j - 1$ \Comment{update cluster sizes}\label{line:bal_bell_ford_update_s}
            \State $m_j \gets m_i$ \Comment{switch node $j$ to the same cluster as $i$}\label{line:bal_bell_ford_update_m}
            \State $d_j \gets d_i + W_{i,j}$ \Comment{use the distance via node $i$}\label{line:bal_bell_ford_update_d}
            \State $n_i \gets n_i + 1$, $n_{p_j} \gets n_{p_j} - 1$ \Comment{update predecessor counts}\label{line:bal_bell_ford_update_n}
            \State $p_j \gets i$ \Comment{predecessor of node $j$ is now $i$}\label{line:bal_bell_ford_update_p}
            \State $\text{done} \gets \text{false}$ \Comment{change was made; do not terminate}
        \EndIf
        \State $t \gets t + 1$; $z^{(t)} \gets z$ for all variables $z$ \Comment{only for use in proofs}
   \EndFor
\Until{$t=T_{\text{BFmax}}$ or done}
\State $T \gets t$ \Comment{only for use in proofs}
\State \textbf{return} $m,d,p,n,s$
\EndFunction
\end{algorithmic}
\end{algorithm}

The Lloyd-based algorithms in this paper rely on an initial seeding, which is
an initial set of centers~---~see~\cref{alg:bal_lloyd_init}.  For our
experiments, we use a random selection of centers to seed the clustering and
note that the final clustering depends on this selection.  Adaptive sampling,
$k$-means++~\cite{kmeanspp}, and other methods~\cite{theoreticalkmeans} can be used to improve
the quality of the initial clusters, thereby accelerating convergence and often
leading to higher quality final clusterings as a result.  These methods can be
applied here, but we limit our attention to random initial seedings and focus
on the overarching issue of imbalance in the clusters, which can still persist
even with use of specialized methods for initial seedings.

Finally,  we note that on~\cref{line:bal_bell_ford_cond2a} of~\cref{alg:bal_bell_ford}, the condition $d_i + W_{i,j} = d_j$ should be implemented using an approximate comparison if floating point arithmetic is being used.

\begin{algorithm}[t]
  \caption{Update center nodes to be the cluster centroids. See \cref{tab:symbols} for variable definitions.}\label{alg:center_nodes}
\begin{algorithmic}[1]
\Function{center-nodes}{$W,m,c,d,p,n,D,P$}
    \State $V_a \gets \{i \mid m_i = a\}$ for all $a = 1,\ldots,N_{\rm cluster}$ \Comment{nodes in cluster $a$}

    \For{$a = 1,\ldots,N_{\rm cluster}$} \Comment{treat each cluster $a$ separately}
        \For{$i \in V_a$}
            \State $q_i \gets \sum_{j \in V_a} (D_{i,j})^2$ \Comment{sum of squared distances to other cluster nodes}
        \EndFor

        \State $i \gets c_a$ \Comment{current cluster center}
        \For{$j \in V_a$} \Comment{test node $j$ as a new cluster center}
            \If{$q_j < q_i$} \Comment{does node $j$ have a \textit{strictly} better metric?}
                \State $i \gets j$ \Comment{$j$ will be the new cluster center}
            \EndIf
        \EndFor

        \If{$i \ne c_a$} \Comment{have we found a new center?}
            \State $c_a \gets i$ \label{line:center_nodes_update_c}
            \State $n_j \gets 0 \text{ for all } j \in V_a$ \Comment{reset predecessor counts}
            \For{$j \in V_a$} \Comment{update data for all nodes in the cluster}
                \State $d_j \gets D_{i,j}$ \label{line:center_nodes_update_d}
                \State $p_j \gets P_{i,j}$ \label{line:center_nodes_update_p}
                \State $n_{p_j} \gets n_{p_j} + 1$
            \EndFor
        \EndIf

    \EndFor
    \State \textbf{return} $c,d,p,n$
\EndFunction
\end{algorithmic}
\end{algorithm}

\subsection{Rebalancing clustering}\label{sec:rebalance}

Balanced Lloyd clustering improves the uniformity and roundness of
the clusters in the Lloyd clustering. Yet there there is no
guarantee that the energy is minimized \textit{across} clusters,
due to the initial seeding.  In this section, we develop a rebalancing
algorithm that reduces the overall energy in the clustering by \textit{splitting}
clusters into two and reducing energy, and by \textit{eliminating} clusters leading to an increase in energy.
The trade-off maintains a constant number of clusters, but reduces the total energy in the clustering.
The rebalanced Lloyd algorithm is given in~\cref{alg:rebal_lloyd_agg} and the rebalancing algorithm itself is given in~\cref{alg:rebalance}.

\begin{algorithm}[!ht]
\caption{Rebalanced version of Lloyd clustering. See \cref{tab:symbols} for variable definitions.}\label{alg:rebal_lloyd_agg}
\begin{algorithmic}[1]
  \Function{rebalanced-lloyd-clustering}{$W,c,T_{\text{max}},T_{\text{BFmax}}$}
  \State $t=0$
  \Repeat
      \State $m, c, d, p, n, s, D, P \gets \textsc{balanced-lloyd-clustering}(W, c, T_{\text{max}},T_{\text{BFmax}}$)
      \State $c \gets \textsc{rebalance}(W, m, c, d, p, D)$
      \State $t = t+1$
  \Until{$t=T_{\text{max}}$ or no change in $c$}
  \State \textbf{return} $m, c, d, p, n, s, D, P$
  \EndFunction
\end{algorithmic}
\end{algorithm}

The algorithm relies on two calculations, the first being
in~\cref{alg:elimination_penalty}, which iterates through each cluster and
calculates the energy penalty (increase) resulting from eliminating a cluster
and merging each node with its nearest cluster. The nearest cluster of a node
is defined based on the distance of the centre of the cluster from the node.
Similarly, \cref{alg:split_improvement} computes the energy improvement
(decrease) from optimally splitting each cluster into two clusters. Here, we
determine the splitting (of each cluster) that results in the lowest energy by
considering all possible pairs of new centers within the cluster.

With measures on the penalties and improvements in energy, the rebalancing
algorithm proceeds by eliminating and splitting clusters in pairs, thereby
reducing the total energy while keeping the number of clusters constant. At
first, it eliminates the cluster with the smallest elimination penalty and splits
the cluster with the largest split improvement, if these are distinct clusters.
It then proceeds to eliminate the cluster with the second-smallest penalty and
split the one with the second-largest improvement, again assuming they are
distinct. This process continues until the energy will no longer be decreased
(i.e., the next elimination penalty would be greater than or equal to the next
split improvement), at which point rebalancing terminates. To access
the clusters in sorted order we use an $\textsc{argsort}(L)$ function that
returns the array of indexes $[i_1, i_2, \ldots]$ so that $L_{i_1}, L_{i_2},
\ldots$ will be in sorted order.

During rebalancing, we assume that the elimination penalties and split
improvements of the clusters do not change as we are actually eliminating and
splitting other clusters. However, the penalty and improvement of cluster $a$
depend on its neighboring clusters. For this reason, when we eliminate or split
a cluster, we mark all of its neighbors as unavailable for being eliminated or
split themselves. This ensures that the penalties and improvement values remain
correct for all clusters that are under consideration for elimination or
splitting at each step.
\begin{algorithm}[t]
\caption{Rebalance clusters by eliminating low-energy clusters and splitting the same number of high-energy clusters in two. See \cref{tab:symbols} for variable definitions.}\label{alg:rebalance}
\begin{algorithmic}[1]
\Function{rebalance}{$W,m,c,d,p,D$}
    \State $V_a \gets \{i \mid m_i = a\}$ for all $a = 1,\ldots,N_{\rm cluster}$ \Comment{nodes in cluster $a$}
    \State $L \gets \textsc{elimination-penalty}(W, m, d, D)$
    \State $(S,c^1,c^2) \gets \textsc{split-improvement}(m, d, D)$
    \State $M_a \gets \textsc{true}$ for all $a \gets 1,\ldots,N_{\rm cluster}$ \Comment{all clusters are modifiable}

    \State $L^{\rm sort} \gets \textsc{argsort}(L)$ \label{line:rebalance_argsort_L} \label{line:rebalance_start_block}
    \State $S^{\rm sort} \gets \textsc{argsort}(S)$ \label{line:rebalance_argsort_S}
    \State $i_{\rm L} \gets 1$ \Comment{sorted index of cluster to eliminate}
    \State $i_{\rm S} \gets N_{\rm cluster}$ \Comment{sorted index of cluster to split}

    \While{$i_{\rm L} \le N_{\rm cluster}$ and $i_{\rm S} \ge 1$}

    \State $a_{\rm L} \gets L^{\rm sort}_{i_{\rm L}}$ \Comment{cluster to eliminate}
    \State $a_{\rm S} \gets S^{\rm sort}_{i_{\rm S}}$ \Comment{cluster to split}

    \If{not $M_{a_{\rm L}}$ or $a_{\rm L} = a_{\rm S}$} \Comment{is cluster $a_{\rm L}$ modifiable and distinct?}
    \State $i_{\rm L} \gets i_{\rm L} + 1$
    \State {\bf continue}
    \EndIf

    \If{not $M_{a_{\rm S}}$} \Comment{is cluster $a_{\rm S}$ modifiable?}
    \State $i_{\rm S} \gets i_{\rm S} - 1$
    \State {\bf continue}
    \EndIf

    \If{$L_{a_{\rm L}} \ge S_{a_{\rm S}}$}\Comment{will the energy \textit{not decrease}?}\label{line:rebalance_compare_L_S}
    \State {\bf break}
    \EndIf

    \State $\textsc{mark-unavailable}(a_L,m,M,W,V_{a_{\rm L}})$
    \State $\textsc{mark-unavailable}(a_S,m,M,W,V_{a_{\rm S}})$

    \State $c_{a_{\rm L}} \gets c^1_{a_{\rm S}}$ \Comment{eliminate cluster $a_{\rm L}$}
    \State $c_{a_{\rm S}} \gets c^2_{a_{\rm S}}$ \Comment{split cluster $a_{\rm S}$}

    \EndWhile
    \State \textbf{return} $c$
\EndFunction
\end{algorithmic}
\end{algorithm}

\begin{algorithm}[t]
\caption{Calculate the energy increase that would result from eliminating each cluster. See \cref{tab:symbols} for variable definitions.}\label{alg:elimination_penalty}
\begin{algorithmic}[1]
\Function{elimination-penalty}{$W,m,d,D$}
    \State $V_a \gets \{i \mid m_i = a\}$ for all $a = 1,\ldots,N_{\rm cluster}$ \Comment{nodes in cluster $a$}
    \For{$a = 1,\ldots,N_{\rm cluster}$}
    \State $L_a \gets 0$ \Comment{energy penalty for eliminating cluster $a$}
    \For{$i \in V_a$}
    \State $d_{\rm min} \gets \infty$ \Comment{minimum distance to a different cluster center}
    \For{$j \in V_a$} \Comment{look for connectivity via $j$}
    \For{$k$ such that $W_{k,j} > 0$} \Comment{all neighbors of $j$}
    \If{$m_k \ne m_j$} \Comment{is $k$ in a different cluster to $j$?}
    \If{$d_k + W_{k,j} + D_{j,i} < d_{\rm min}$} \Comment{is $k$'s center closer?}
    \State $d_{\rm min} \gets d_k + W_{k,j} + D_{j,i}$
    \EndIf
    \EndIf
    \EndFor
    \EndFor
    \State $L_a \gets L_a + (d_{\rm min})^2$ \Comment{add the new energy for $i$}
    \EndFor
    \State $L_a \gets L_a - \sum_{i \in V_a} (d_i)^2$ \Comment{subtract the current energy metric}
    \EndFor
    \State \textbf{return} $L$
\EndFunction
\end{algorithmic}
\end{algorithm}

\begin{algorithm}[t]
\caption{Calculate the energy decrease that would result from optimally splitting each cluster in two. See \cref{tab:symbols} for variable definitions.}\label{alg:split_improvement}
\begin{algorithmic}[1]
\Function{split-improvement}{$m,d,D$}
    \State $V_a \gets \{i \mid m_i = a\}$ for all $a = 1,\ldots,N_{\rm cluster}$ \Comment{nodes in cluster $a$}
    \For{$a = 1,\ldots,N_{\rm cluster}$}
    \State $S_a \gets \infty$ \Comment{energy improvement for splitting cluster $a$}
    \For{$i \in V_a$} \Comment{first possible new center}
    \For{$j \in V_a$} \Comment{second possible new center}
    \State $S_{\rm new} \gets 0$ \Comment{energy with centers $i$ and $j$}
    \For{$k \in V_a$} \Comment{compute cost for node $k$}
    \If{$D_{i,k} < D_{j,k}$} \Comment{is $k$ closer to center $i$ or $j$?}
    \State $S_{\rm new} \gets S_{\rm new} + (D_{i,k})^2$
    \Else
    \State $S_{\rm new} \gets S_{\rm new} + (D_{j,k})^2$
    \EndIf
    \EndFor %
    \If{$S_{\rm new} < S_a$} \Comment{is this a better split?}
    \State $S_a \gets S_{\rm new}$ \Comment{store the new energy}
    \State $c^1_a \gets i$ \Comment{store the new centers $i$ and $j$}
    \State $c^2_a \gets j$
    \EndIf
    \EndFor %
    \EndFor %
    \State $S_a \gets \sum_{i \in V_a} (d_i)^2 - S_a$ \Comment{improvement from current cluster energy}
    \EndFor %
    \State \textbf{return} $S,c^1,c^2$
\EndFunction
\end{algorithmic}
\end{algorithm}

\begin{algorithm}[t]
\caption{Mark a cluster and all of its neighbors as unavailable. See \cref{tab:symbols} for variable definitions.}\label{alg:mark_unavailable}
\begin{algorithmic}[1]
\Function{mark-unavailable}{$a,m,M,W,V_a$}
    \State $M_a \gets \text{False}$ \Comment{cluster $a$ is unavailable}
    \For{$i \in V_a$}
    \For{$j$ such that $W_{i,j} > 0$} \Comment{all neighboring nodes of cluster $a$}
    \State{$M_{m_j} \gets \text{False}$} \Comment{cluster of node $j$ is unavailable}
    \EndFor %
    \EndFor %
\EndFunction
\end{algorithmic}
\end{algorithm}

\subsection{Theoretical observations}

In this section we provide theoretical observations on the balanced and rebalanced Lloyd algorithms, covering three key aspects: (1) connectedness of the resulting clusters, (2) the energy-decreasing behavior of the algorithms, and (3) the computational complexity of the algorithms. We note that none of these algorithms are guaranteed to find the global minimum of the energy, but we show that they do reduce the energy at each non-terminating step, and they are guaranteed to terminate.

We have formulated the balanced Bellman-Ford algorithm with a cap on the maximum number of iterations, which will be necessary for proving linear complexity in \cref{thm:bal_lloyd_agg_complexity}. In practice $T_{\text{BFmax}}$ can be chosen to be the maximum expected cluster radius and implementations should warn if \cref{alg:bal_bell_ford} reaches this limit, as this may indicate that the clusters are not connected. This observation relies on the following result.

\begin{theorem}\label{thm:bal_bell_ford_conn}
  The clusters returned by \cref{alg:bal_bell_ford} are connected if it terminates before the maximum number of iterations.
\end{theorem}
\begin{proof}
  The proof relies on understanding the state of the variables within~\cref{alg:bal_bell_ford} as the algorithm iterates.  We denote this by using a superscript, $z^{(t)}$ to indicate the state of variable $z$ at a given ``time'', $t$, in the algorithm, with time $T$ denoting completion.

  We wish to show that, for each $j$, $m_j^{(T)} = m_i^{(T)}$ where $i = p_j^{(T)}$, which ensures that the predecessor-paths to cluster centers are contained within each cluster.
  Suppose not and let $t$ be the last iteration when $m_j$ and $p_j$ were updated by \cref{line:bal_bell_ford_update_m,line:bal_bell_ford_update_p}.
  Taking $i = p_j^{(t)}$ we have $m_i^{(t)} = m_j^{(t)} = m_j^{(T)} \ne m_i^{(T)}$ so there must be a later $t' > t$ at which $m_i$ was updated for the last time.
  At this later time, we must have at least one of the following cases.

  \textit{Case 1: condition on \cref{line:bal_bell_ford_cond1} is true.}
  Then $d_i^{(t')} < d_i^{(t)}$, and so $d_i^{(t')} + W_{i,j} < d_i^{(t)} + W_{i,j} = d_j^{(t)} = d_j^{(T)}$.
  This is a contradiction because we cannot have $d_j > d_i + W_{i,j}$ when the algorithm terminates.

  \textit{Case 2: conditions on \cref{line:bal_bell_ford_cond2a,line:bal_bell_ford_cond2c} are true.}
  Then $n_i^{(t')} = 0$, which is impossible since $p_j^{(t')} = p_j^{(t)} = i$.
\end{proof}

To understand the behavior of the balanced algorithms, we consider a modified energy that includes a second term for the cluster sizes. Define
\begin{equation}
  \label{eq:squared_energy}
  H_\delta = \sum_{i = 1}^{N_{\rm node}} (d_i)^2
  + \delta \sum_{a = 1}^{N_{\rm cluster}} (s_a)^2,
\end{equation}
where $d_i$ is the distance from node $i$ to its cluster center, $s_a = |\{i \mid m_i = a\}|$ is the size (number of nodes) of cluster $a$, and
\begin{equation}
  \label{eq:delta}
  \delta = \left(\frac{\Delta_{\rm min}}{N_{\rm node}}\right)^2
\end{equation}
is chosen based on the minimum difference, $\Delta_{\rm min}$, between distinct values of $W_{i,j}$
(or an arbitrarily small positive number if there are no distinct values of $W_{i,j}$).
The first term in \cref{eq:squared_energy} is the sum of squared distances from nodes to their cluster centers, while the second term is the sum of squared cluster sizes. Note that $\delta$ is chosen so that the second term in \cref{eq:squared_energy} is always less than the minimum possible increment in the first term.

\begin{lemma}\label{lem:bal_bf_energy_decrease}
  \Cref{alg:bal_bell_ford} results in a decrease of the energy~\eqref{eq:squared_energy}, or preserves the energy if no change is made to the clustering.
\end{lemma}

\begin{proof}
  We will show that all steps in the algorithm that change $d_j$ or $s_a$ result in a strict decrease of $H_\delta$.

  \textit{Case 1:} updates by \cref{line:bal_bell_ford_update_s,line:bal_bell_ford_update_d} with $d_j$ strictly decreasing.
  Then the reduction in the first term in $H_\delta$ is at least $(\Delta_{\rm min})^2$ and any increase in the second term in $H_\delta$ is less than $(N_{\rm node})^2$, so the definition of $\delta$ means the decrease strictly dominates.

  \textit{Case 2:} updates by \cref{line:bal_bell_ford_update_s,line:bal_bell_ford_update_d} with $d_j$ constant and $\mathfrak{s}_j > \mathfrak{s}_i + 1$.
  Then the first term in $H_\delta$ is constant and $s_{m_i} \gets \mathfrak{s}_i + 1$ and $s_{m_j} \gets \mathfrak{s}_j - 1$ results in a strict decrease of $(s_{m_i})^2 + (s_{m_j})^2$.
\end{proof}

\begin{lemma}\label{lem:center_nodes_energy_decrease}
  \Cref{alg:center_nodes} results in a decrease of the energy~(\ref{eq:squared_energy}), or preserves the energy if no change is made to the clustering.
\end{lemma}

\begin{proof}
  Only updates by \cref{line:center_nodes_update_d} will change $d_j$. Because the change is caused by the use of $j$ as the new center, and $q_j < q_i$, the first term in $H_\delta$ strictly decreases and the second term is unchanged.
\end{proof}

\begin{theorem}\label{thm:bal_lloyd_agg_termination}
  \Cref{alg:bal_lloyd_agg} terminates, even if $T_{\text{max}} = \infty$.
\end{theorem}

\begin{proof}
  From \cref{lem:bal_bf_energy_decrease,lem:center_nodes_energy_decrease}, all steps in the algorithms that change $d_i$ or $s_a$ result in a strict decrease of $H_\delta$. Because $H_\delta$ is positive and can only take a finite number of values, and we terminate when no changes are made, this ensures termination.
\end{proof}

\begin{theorem}
  \Cref{alg:rebalance} results in a decrease or preservation of the energy \cref{eq:squared_energy}.
\end{theorem}

\begin{proof}
Because of~\cref{line:rebalance_compare_L_S}, each elimination/split pairing explicitly results in a decrease of the first term in \cref{eq:squared_energy} and thus also a decrease in the overall value of $H_\delta$ due to the choice of $\delta$.
\end{proof}

To give bounds on the computational complexity of the algorithms, we require the following assumptions on the graph structure.

\begin{assumption}\label{ass:graph_complexity}
  Assume that the number of edges in $G$ incident on each vertex in the graph is bounded independently of $\Nnode$, and that the initial centers, $c$, are such that clusters found by~\cref{alg:bal_bell_ford} have size bounded independently of $\Nnode$.
\end{assumption}

We note that the first part of~\cref{ass:graph_complexity} is reasonable for many application areas, including when considering graphs associated with the sparse matrices that arise from finite-element (and other) discretizations of PDEs.  Here, it is uncommon to see vertices with more than $\mathcal{O}(1)$ incident edges.  The second part of~\cref{ass:graph_complexity} arises from a probabilistic lens, where we assume that the initial random seeding is ``well-distributed'', so that there are not large connected components with no initial centers, leading to large distances from $\mathcal{O}(\Nnode)$ points to their nearest center.  Quantifying the probability with which this holds depends strongly on the properties of the graphs under consideration, and we leave this aspect for future work, noting the worst-case complexity of our algorithm is certainly quadratic (or worse) in $\Nnode$.  Achieving good clustering in linear time without a similar assumption on the initial seeds has also been investigated, necessarily involving long-range exchanges to move centers from regions of the graph with ``extra'' centers to regions with many points far from a center~\cite{doi:10.1137/040607769}.

\begin{theorem}\label{thm:bal_lloyd_agg_complexity}
  Under \cref{ass:graph_complexity}, if $T_{\text{max}}$ and $T_{\text{BFmax}}$ are both bounded independently of $\Nnode$, then the total cost of~\cref{alg:bal_lloyd_agg} is $\mathcal{O}(\Nnode)$.
\end{theorem}

\begin{proof}
  This follows from cost estimates for each of the components of~\cref{alg:bal_lloyd_agg}.
  The inner loop of~\cref{alg:bal_bell_ford} has complexity equal to the number of edges in $G$, which is $\mathcal{O}(\Nnode)$ by~\cref{ass:graph_complexity}.  If $T_{\text{BFmax}} = \mathcal{O}(1)$, then~\cref{alg:bal_bell_ford} has complexity $\mathcal{O}(\Nnode)$.  The cost of Floyd-Warshall on each cluster is cubic in the cluster size, which we assume (as a function of the initial centers and clusters) to be $\mathcal{O}(1)$, giving a total cost of $\mathcal{O}(\Ncluster) = \mathcal{O}(\Nnode)$.  \Cref{alg:center_nodes} also has linear complexity.  Since $T_{\text{max}}$ in~\cref{alg:bal_lloyd_agg} is also $\mathcal{O}(1)$, the total complexity is, then, $\mathcal{O}(\Nnode)$.
\end{proof}

\begin{theorem}
  Under \cref{ass:graph_complexity}, \cref{alg:rebalance} terminates with cost at most $\mathcal{O}(N_{\rm node} \log N_{\rm node})$ and, if $T_{\text{max}}$ is bounded independently of $\Nnode$, then \cref{alg:rebal_lloyd_agg} also has $\mathcal{O}(N_{\rm node} \log N_{\rm node})$ total cost.
\end{theorem}

\begin{proof}
  From \cref{thm:bal_lloyd_agg_complexity}, the cost of \cref{alg:bal_lloyd_agg} is $\mathcal{O}(N_{\rm node})$. Both \cref{alg:elimination_penalty} and \cref{alg:split_improvement} iterate over all clusters and perform bounded work per cluster (using the bound on cluster size from \cref{ass:graph_complexity}), so they have cost $\mathcal{O}(\Ncluster) = \mathcal{O}(N_{\rm node})$. Similarly, \cref{alg:mark_unavailable} has cost independent of $\Nnode$ because it only iterates over nodes within a single cluster. The cost of \cref{alg:rebalance} is thus $\mathcal{O}(N_{\rm node} \log N_{\rm node})$ because \cref{line:rebalance_argsort_L,line:rebalance_argsort_S} are $\mathcal{O}(N_{\rm cluster} \log N_{\rm cluster})$ and all other loops and subroutines are $\mathcal{O}(N_{\rm node})$. Finally, assuming $T_{\text{max}} = \mathcal{O}(1)$, we have the same cost for \cref{alg:rebal_lloyd_agg}.
\end{proof}

\begin{remark}
  The algorithmic complexity of \cref{alg:rebalance} and, hence, that of \cref{alg:rebal_lloyd_agg}, can be reduced to $\mathcal{O}(N_{\rm node})$ by changing the algorithm to separately treat fixed-size sets of clusters. That is, rather than considering all clusters at once, partition the set of clusters into subsets and run \cref{alg:rebalance} separately on each subset. This will avoid the $\mathcal{O}(N_{\rm node} \log N_{\rm node})$ sorts in~\cref{line:rebalance_argsort_L,line:rebalance_argsort_S}, leaving the cost as linear in $N_{\rm node}$. We expect this would lead to some slight reduction in the quality of the rebalance, because eliminate/split pairings will only be considered within a subset but, for large subsets, we would not expect this to make a significant difference. This subset approach is also the natural way to parallelize \cref{alg:rebalance}, with one subset per processor.
\end{remark}

\begin{remark}
Parallelization of~\cref{alg:bal_lloyd_agg,alg:rebal_lloyd_agg} relies on parallelization of the other underlying algorithms.  Both~\cref{alg:floyd_warshall,alg:center_nodes} operate independently on each cluster and are, thus, naturally parallelizable.  \Cref{alg:bal_bell_ford} could be naturally parallelized by applying it independently to the set of nodes owned by each processor in a parallel decomposition.
\end{remark}

\section{Numerical Results}\label{sec:numerics}

In this section, we highlight the value of balanced Lloyd clustering with
rebalancing for smoothed aggregation multigrid.  All
computations are performed with PyAMG~\cite{BeOlSc2022}.
Unless stated otherwise, all results below consider a standard Poisson problem of form
\begin{subequations}\label{eq:modelproblem}
  \begin{align}
    -\nabla\cdot\nabla U & = F\quad\text{in $\Omega$},\\
    \vec{n}\cdot \nabla U & = 0\quad\text{on $\partial\Omega$},
  \end{align}
\end{subequations}
where Neumann boundary conditions are used to highlight clustering near the
boundary.\footnote{With Dirichlet conditions, we would observe ``singleton''
  clusters for the isolated points.  This does not impact the method, only
visualization.}  \Cref{eq:modelproblem} is discretized using either standard $P^1$
linear finite elements on a triangulation of the domain, $\Omega$, or $Q^1$ bilinear finite elements on a quadrilateral mesh of $\Omega$, yielding a matrix problem of the form
\begin{equation}\label{eq:linearsystem}
  A u = f.
\end{equation}
In the following convergence tests, $f$ is set to zero and a random approximation
to $u$ is used to initialize the AMG cycling.

We consider three main cases of clustering in the context of AMG\@:
standard Lloyd clustering (\cref{alg:lloyd_agg}),
balanced Lloyd clustering (\cref{alg:bal_lloyd_agg}), and
balanced Lloyd clustering with rebalancing (\cref{alg:rebalance}).
For each of these, we require a definition of the weight matrix, $W$, and the
number of clusters, $\Ncluster$.  In each case, we bound the number of inner
iterations of Lloyd clustering at five and the number of rebalance sweeps at
four; in practice, this is a conservative bound and the iterations complete much
earlier (due to no change in the clustering state).

To form the weight matrix, $W$, we consider the so-called \textit{evolution}
measure~\cite{2010_OlScTu_evosoc} which associates a value of strength
for each edge in the graph of $A$ in~\cref{eq:linearsystem} based on smoothing
properties. This leads to an initial non-negative weight matrix, $\widehat{W}$, where a large edge value
$\widehat{W}_{i,j}$ indicates that nodes $i$ and $j$ should be clustered together.
The algorithms above make use of an assumption that the graph associated with $W$ is connected, but this is not guaranteed to be the case for that of $\widehat{W}_{i,j}$.   Thus, we augment $\widehat{W}$
with a small padding for each edge in $A$, defining $\widetilde{W}$ as
\begin{equation}
  \widetilde{W}_{i,j} \gets \widehat{W}_{i,j} + 0.1\quad\text{if $A_{i,j}\ne 0$}
\end{equation}
The Lloyd-based clustering presented here is based on shortest distances in the graph of $W$.  As a
result, we consider the inverse strength as a proxy for distance.  This results in defining $W$ so that
\begin{equation}
  W_{i,j} = \frac{1}{\widetilde{W}_{i,j}} \text{ if }\widetilde{W}_{i,j}\neq 0,
\end{equation}
so that strong edges refer to shorter distances.  With this inversion, the additional padding added above
indicates a long distance in the weight matrix.

\subsection{Varying cluster numbers}

While Greedy and MIS-based clustering have been used successfully in
many settings, they do not provide a mechanism to control the number
of resulting clusters.  Here,
we explore the ability of Lloyd clustering to target
specific numbers of clusters.
As motivation, consider the model problem on an unstructured triangulation of the unit disk with \num{10245} vertices and
\num{20158} elements.  We construct multigrid hierarchies using rebalanced Lloyd clustering, setting
the target number of points in each cluster at each level to a fixed value between 3
and 20.  We estimate the asymptotic convergence
factor by the geometric mean of the last five residual norms at convergence,
say $k$ iterations:
\begin{equation}
  \rho = \left(\frac{\|r^{(k)}\|}{\|r^{(k-5)}\|}\right)^{\frac{1}{5}},
\end{equation}
where $r^{(k)} = f - Au^{(k)}$ is the residual vector after $k$ iterations.
Combined with a model for the \textit{cost} of each multigrid cycle, given by
the total number of non-zeros in the sparse-matrix operations in the cycle
(i.e., the \textit{cycle complexity}, $\chi$), this leads to a measure of the work per
digit of accuracy (WPD) for the method:
\begin{equation}
  \text{WPD} = \frac{\chi}{-\log_{10}(\rho)}.
\end{equation}
\Cref{fig:disc_wpd_stats} shows that the efficiency (and effectiveness) of an AMG method
can vary depending on the (average) number of points per cluster; in this case,
we observe that very small clusters lead to rapid convergence (small $\rho$), yet
due to the slower coarsening, the total complexity of the multigrid cycle is higher.
\begin{figure}
  \centering
  \includegraphics{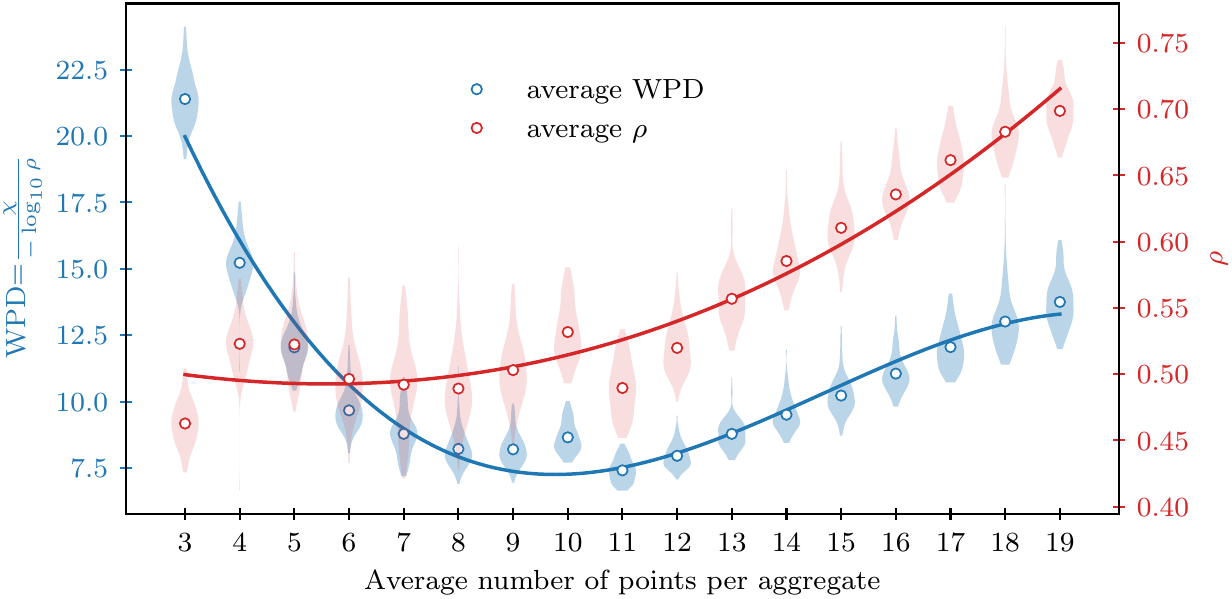}
  \caption{Work per digit (WPD) of accuracy and convergence $\rho$
for clustering sizes ranging from 3--19 points per cluster (on average)
using rebalanced Lloyd clustering.
The average over 100 runs is marked $\circ$ and a trendline from
a smoothed cubic spline is given for the mean (solid).}\label{fig:disc_wpd_stats}
\end{figure}

In the end, balanced Lloyd clustering with rebalancing leads to well-formed clusters
and the ability to use a vast range of cluster sizes.  \Cref{fig:disc_varynaggs} illustrates
a range of cases for a smaller mesh of the same domain, from five (large) clusters at one extreme to 250 small clusters including
singleton and many pairwise clusters.  True pairwise clustering~\cite{doi:10.1137/100818509, doi:10.1137/120876083} is
not represented; however, it remains an open question whether a Lloyd-type algorithm could
render nearly pairwise clustering using modified criteria for tiebreaking and rebalancing.
\begin{figure}
  \centering
  \includegraphics{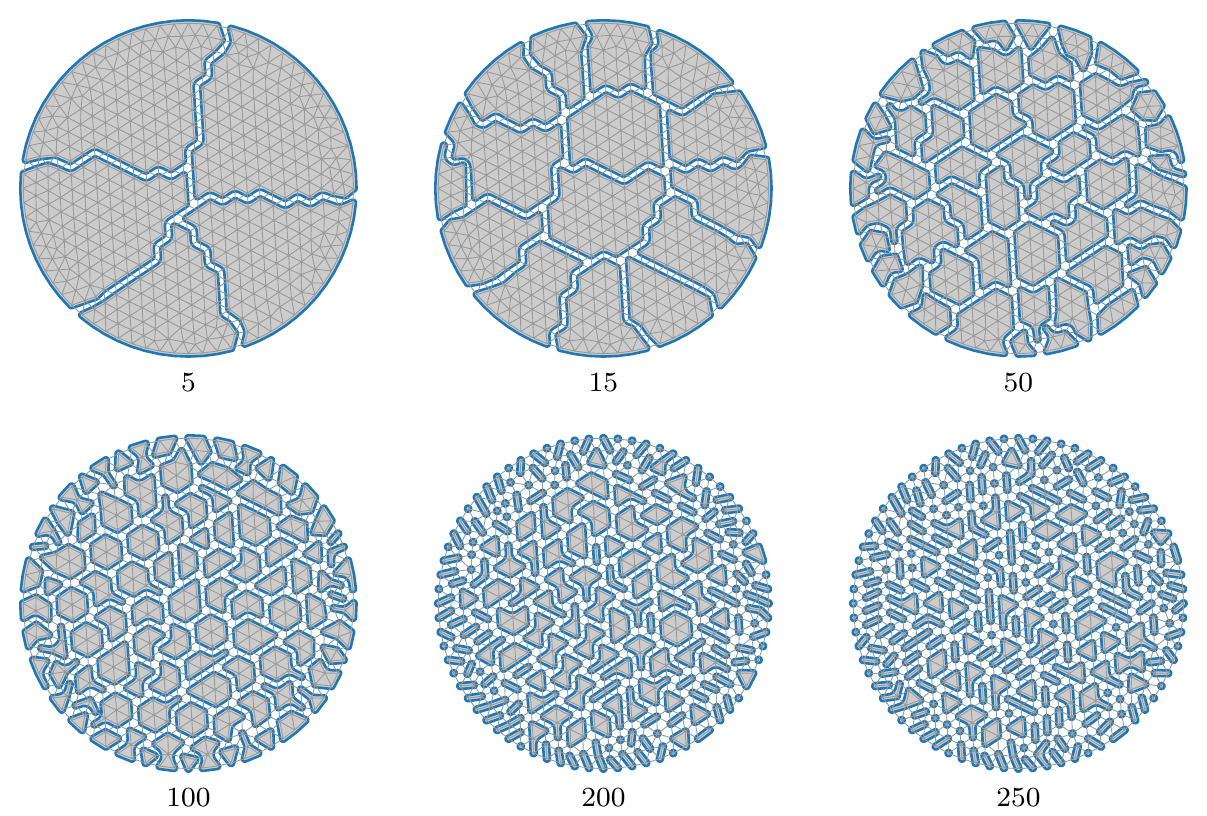}
  \caption{Example clustering patterns with the number of clusters ranging from 5 to 250
           using rebalanced Lloyd clustering.}\label{fig:disc_varynaggs}
\end{figure}

\subsection{Tiebreaking}

\Cref{alg:bal_bell_ford} introduces
``tiebreaking'' on \cref{line:bal_bell_ford_cond2b}. If a node
in the graph
is equidistant from multiple centers,
then the node becomes a member of the
neighboring cluster if the neighboring cluster is smaller by two in size
than the current cluster of the node.
Tiebreaking in the balanced Bellman-Ford algorithm
impacts the uniformity of the sizes of the clusters.
To quantify uniformity,
we consider discretizing~\cref{eq:modelproblem} on a uniform $64\times64$ quadrilateral mesh.
We cluster the nodes
using the balanced Bellman-Ford algorithm \textit{with} and \textit{without}
tiebreaking, requesting the number of clusters be equal to 10\% of the fine-grid number of nodes (rounded down when this is not an integer). We randomly distribute the initial seeding \num{1000} times
and, in each case, compute the following metrics:
the number of zero diameter clusters (i.e., singleton clusters),
the standard deviation in the number of nodes per cluster,
and the energy for each clustering (defined by~\cref{eq:energy}).

\Cref{fig:square_diamters_zero}
shows the number of clusters having zero diameter with and without tiebreaking,
highlighting that tiebreaking substantially decreases the number of clusterings with zero
diameter clusters, from over one-third of clusterings to about one percent.
Likewise,
\cref{fig:square_diameters_energy} (left)
shows the effect of tiebreaking on the distribution of the standard
deviation in the number of nodes. Here,
tiebreaking leads to a decrease
yielding clusters more uniform in size.
Tiebreaking also contributes to clusters that are more round~---~this is supported by
\cref{fig:square_diameters_energy} (right),
where we see that tiebreaking decreases the energy of the system.  We
emphasize that tiebreaking is an inexpensive strategy that clearly
improves performance of the clustering.
\begin{figure}
  \centering
  \includegraphics{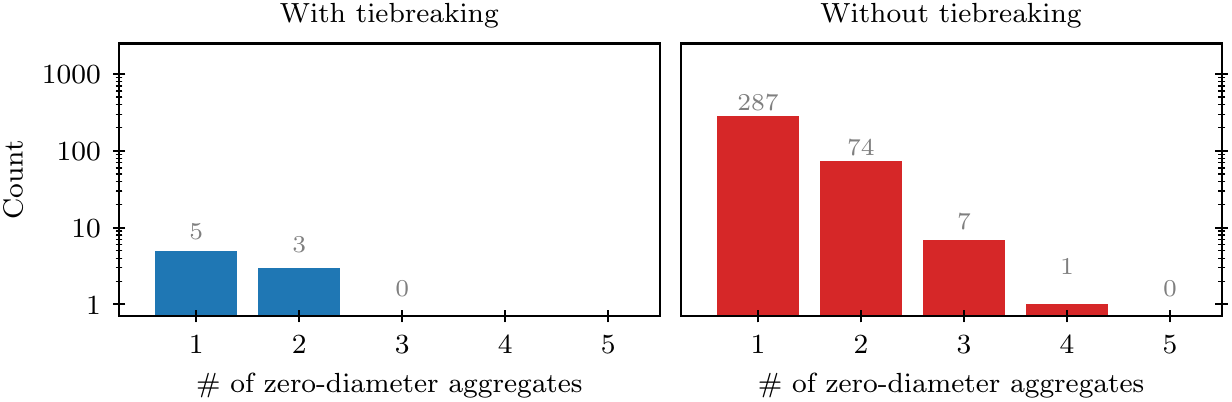}
\caption{Distribution of the number of clusters having zero diameter for
  balanced Lloyd clustering with or without tiebreaking for a $64\times 64$ mesh.}\label{fig:square_diamters_zero}
\end{figure}
\begin{figure}
  \centering
  \includegraphics{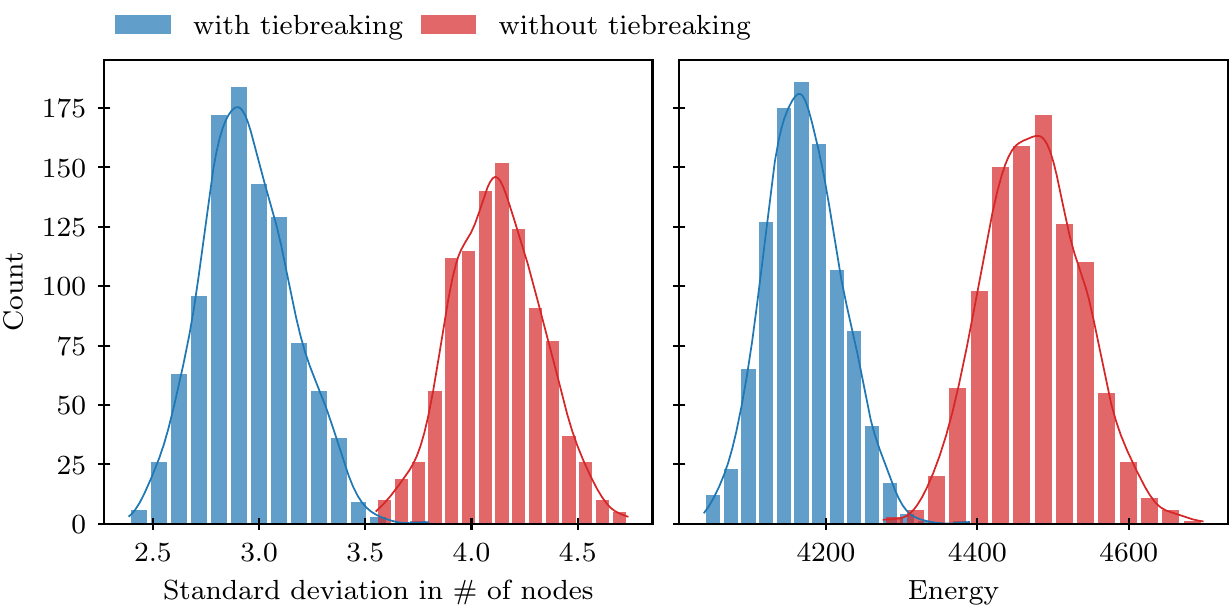}
  \caption{Distribution of the standard deviation in the number of nodes
    and distribution of energy
    for balanced Lloyd clustering with or without tiebreaking on a $64\times 64$ mesh.}\label{fig:square_diameters_energy}
\end{figure}

\subsection{Rebalancing}

To quantify the improvements in cluster quality as we move from standard to balanced, and then to rebalanced Lloyd clustering, we again consider a $64\times64$
quadrilateral mesh. Nodes in this mesh are clustered
using the three
methods, again using 10\% of $\Nnode$ to determine $\Ncluster$. In each case, the clustering is repeated
1000 times, yielding a standard deviation of cluster diameter,
standard deviation of number of nodes in clusters, and
energy for each test. The results are averaged
and the same experiment is
performed for $16\times16$, $32\times32$, and $128\times128$
meshes.

\Cref{fig:square_stats_comparison_64}
shows the distributions for each method in the case of a
$64\times64$
mesh.
Lower standard deviation of diameter and
standard deviation of number of nodes suggest that the clusters
that result
from rebalanced Lloyd are more uniform in shape and
size compared to the other methods.
This is also reflected by the lower energy
for rebalanced Lloyd clustering.
The figure also highlights that \textit{variation} in the metrics is
lower for rebalanced Lloyd, pointing to the consistency in the method over multiple runs.

\Cref{fig:square_stats_comparison_all_n}
shows the difference between the maximum and
minimum diameters and the energy, averaged over 1000 samples,
for each of the clustering methods as we vary problem size.
The figures underscore that rebalanced Lloyd
yields more
uniform, rounded clusters having less energy than the
other two clustering methods as the mesh size grows.
\begin{figure}
  \centering
  \includegraphics{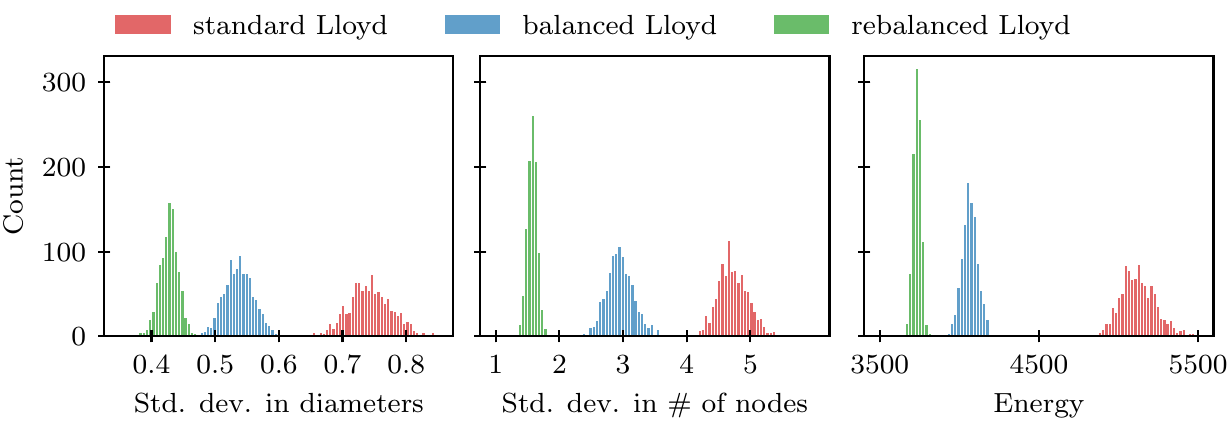}  %
  \caption{Distribution of standard deviation in diameters,
           distribution of standard deviation in number of nodes, and
           distribution in energy
    for different clustering methods for a $64\times 64$ mesh.}\label{fig:square_stats_comparison_64}
\end{figure}
\begin{figure}
  \centering
  \includegraphics{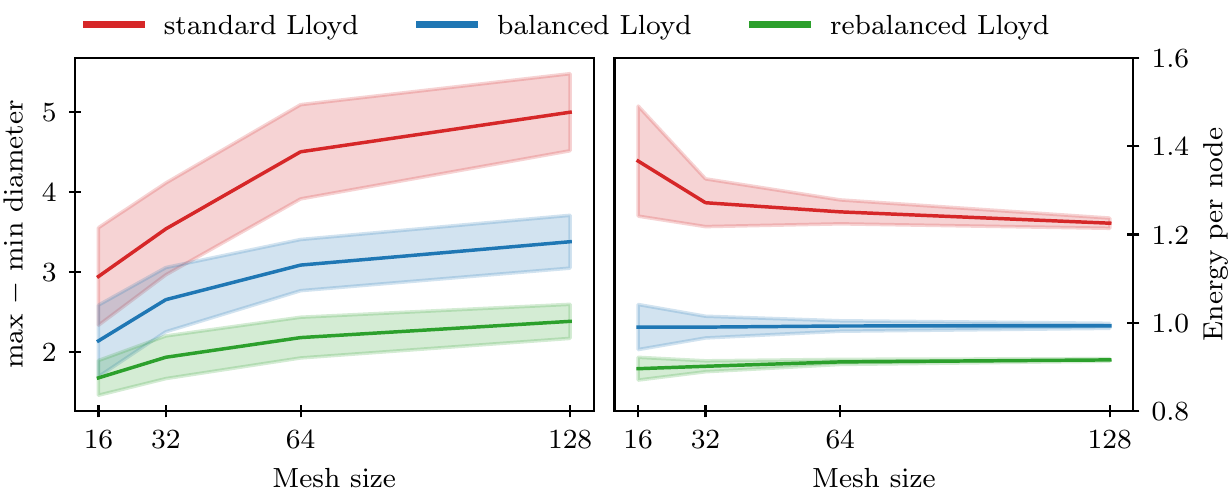}
  \caption{(Left) Difference between maximum and minimum diameters of clusters averaged over 1000 samples; (Right) Energy per node averaged over 1000 samples.  The shaded regions mark one standard deviation from the mean.}\label{fig:square_stats_comparison_all_n}
\end{figure}

\subsection{Impact of seeding}

Consider a 1D (symmetric) graph with unit edge weights and 30 nodes, as shown
in~\cref{fig:oned_example}.  We seek 10 clusters with two different initial
seedings: a worst-case seeding with all initial centers stacked at the left
(that is, $c_i = 0,\ldots,9$) and a random seeding. In this case, we can
identify an optimal clustering of equal-sized clusters (with energy $H=20.0$).

The worst-case seeding requires a very high number of iterations to
substantially reduce the energy, because each iteration only redistributes nodes
between adjacent clusters and we require ``relaxation'' over the entire domain.
Because of this we expect at least $\mathcal{O}(N_{\rm node})$ iterations to
achieve a reasonable clustering for such seedings, for a total cost of
$\mathcal{O}(N_{\rm node}^2)$ or higher. We also see that the balanced algorithm
becomes trapped in a poor local minimum ($H = 113.0$, down from an initial value
of $H = 2870.0$) and requires multiple rebalancing steps to achieve a reasonable
clustering. It is worth noting that even in this extreme case the rebalancing
algorithm does achieve a reasonable final clustering ($H = 29.0$).

In contrast, the random seeding starts with a good coverage of the graph and a
correspondingly lower energy ($H = 38.0$). From such a seeding, the algorithm
terminates rapidly at a reasonable final clustering ($H = 26.0$). We expect that
the number of iterations for such random seedings will be independent of the
number of nodes, for an overall $\mathcal{O}(N_{\rm node})$ cost.
\begin{figure}

  \hrule
  \medskip

  {\scriptsize \textbf{Worst-case seeding}}
  \medskip

  \includegraphics{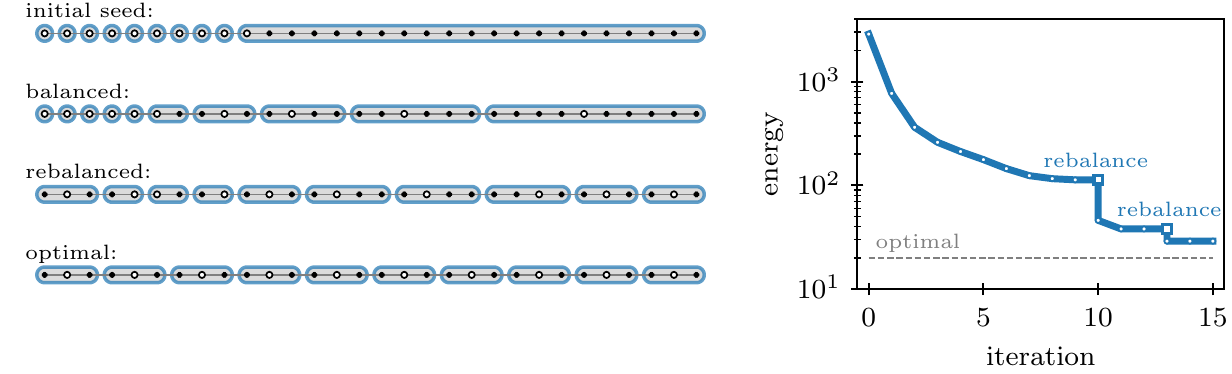}

  \hrule
  \medskip
  {\scriptsize \textbf{Random seeding}}
  \medskip

  \includegraphics{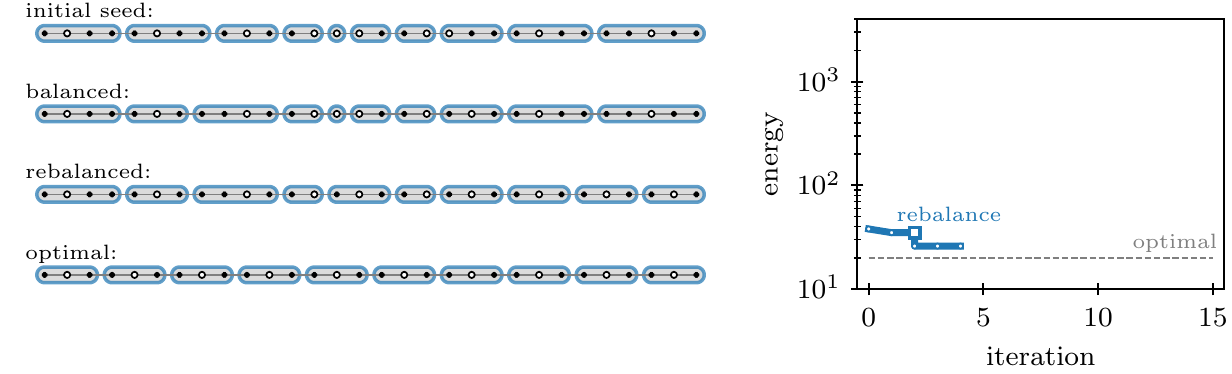}
  \hrule
  \caption{
    Initial seedings with (top) a worst-case ($c_i=0,\ldots,9$) and (bottom) random.
    (left) Cluster centers $c_i$ are given by
    \protect\tikz[baseline=-0.5ex]{\protect\draw[black,fill=white,line width=1.0pt,radius=1.5pt] (0,0) circle;}
    and clusters are marked with
    \protect\tikz[baseline=-0.5ex]{\protect\draw[tab-blue,fill=tab-gray!30,line width=1.0pt] (0,0) circle [x radius=8pt, y radius=4pt];}.  (right) Each curve represents iterations by balanced Lloyd clustering, with rebalance steps marked by
    \protect\tikz[baseline=-0.5ex]{\protect\draw[black,fill=white,line width=1.0pt] (-0.05,-0.05) rectangle (0.05, 0.05);}.
  }\label{fig:oned_example}
\end{figure}

\subsection{Algebraic multigrid convergence}

In the application of clustering to algebraic multigrid (see~\cref{app:AMG}), cluster quality
plays an important role in overall convergence of the method, but one
that is not yet quantified by existing sharp measures.
While we can
easily confirm improvement (or degradation) in the measured
convergence factor after making a change to a clustering, it is
difficult to directly assess if an individual cluster is the cause
of poor convergence.

One way to \textit{localize} a bound on AMG
convergence is to consider the classical bound based on smoothing and
approximation properties~\cite{JWRuge_KStuben_1987a, SMacLachlan_LOlson_2014a}.  This theory
considers the convergence of a two-grid cycle with post-relaxation
given by $u \leftarrow u + M(f-Au)$ and coarse-grid correction given
by $u \leftarrow u + P(P^T A P)^{-1}P^T(f-Au)$.  We write $G = I-MA$ and
$T = I-P(P^T A P)^{-1}P^T A$ as the error-propagation operators of
relaxation and coarse-grid correction, respectively, with the
error-propagation operator of the two-grid scheme given by $GT$.  The
diagonal of SPD matrix $A$ is denoted by $D$.  In what follows, we
assume that $A$ is SPD, $P$ is of full rank, and $\|G\|_A < 1$.
Theorem 4 of~\cite{SMacLachlan_LOlson_2014a} shows that if there exist
constants $\alpha,\beta>0$ such that
\begin{align*}
\|Ge\|_A^2 & \leq \|e\|_A^2 - \alpha \|e\|_{AD^{-1}A}^2 \text{ for all
}e,\\
\text{and }\|Te\|_A^2 & \leq \beta \|Te\|_{AD^{-1}A}^2 \text{ for all }e,
\end{align*}
then $\|GT\|_A \leq (1-\alpha/\beta)^{1/2}$.  The first of these is
known as the \textit{smoothing property,} since it
concerns the action of relaxation, $G$, on errors, $e$.  The second is
referred to as the \textit{approximation property}, since it
quantifies the action of the coarse-grid correction process.  Equations (19) and (20)
of~\cite{SMacLachlan_LOlson_2014a} show that this approximation property is guaranteed by the existence of a constant $\beta > 0$
such that $\inf_{e_c}\|e-Pe_c\|_D^2 \leq \beta\|e\|_A^2$ for all $e$.
Choosing $e_c = (P^T D P)^{-1}P^T De$ and defining $T_D = I -
P(P^T D P)^{-1}P^T D$ then allows us to quantify such a $\beta$ as
\begin{equation}\label{eq:beta_def}
  \beta = \sup_{Ae\neq 0} \frac{ e^T T_D^T D T_D e}{e^T Ae}.
\end{equation}

We find this $\beta$ by solving for the largest eigenvalue of the
generalized eigenvalue problem $T_D^T D T_D e = \lambda Ae$, and let $e$
be the associated eigenvector.  To localize the measure over a single
cluster, we decompose the inner product in the numerator into a sum
over clusters, writing $\beta = \sum_{a=1}^{N_{\rm cluster}}
\beta_a$, where
\begin{equation}\label{eq:betaopt2}
\beta_a = \frac{\left(\sum_{j\in V_a} (D T_D  e)_j (T_D e)_j \right)}{e^T Ae}
\end{equation}
This comes from writing the numerator of~\eqref{eq:beta_def} as the
inner product of $DT_D e$ with $T_D e$, and then localizing the
summation in that inner product over each cluster.

We again consider the Poisson problem on a triangulation of the unit disk with 528 unknowns,
and compute \cref{eq:betaopt2} for each cluster generated by each method.
\Cref{fig:disc_compareenergy_aggregates} shows that the extreme values of $\beta_a$ are reduced through
rebalancing.  Indeed, this is reflected in the convergence shown in
\cref{fig:disc_compareenergy_convergence}, where we observe a dramatic reduction
in the number of iterations for solvers with these clusters.  It is, of course, important to note that not \textit{every}
clustering generated by standard Lloyd exhibits similarly poor performance.  The quality of the initial clustering used to seed the algorithm plays an important role in determining the multigrid performance.  Results seen here are for a representative, randomly generated, initial seeding.
\begin{figure}
  \centering
  \includegraphics{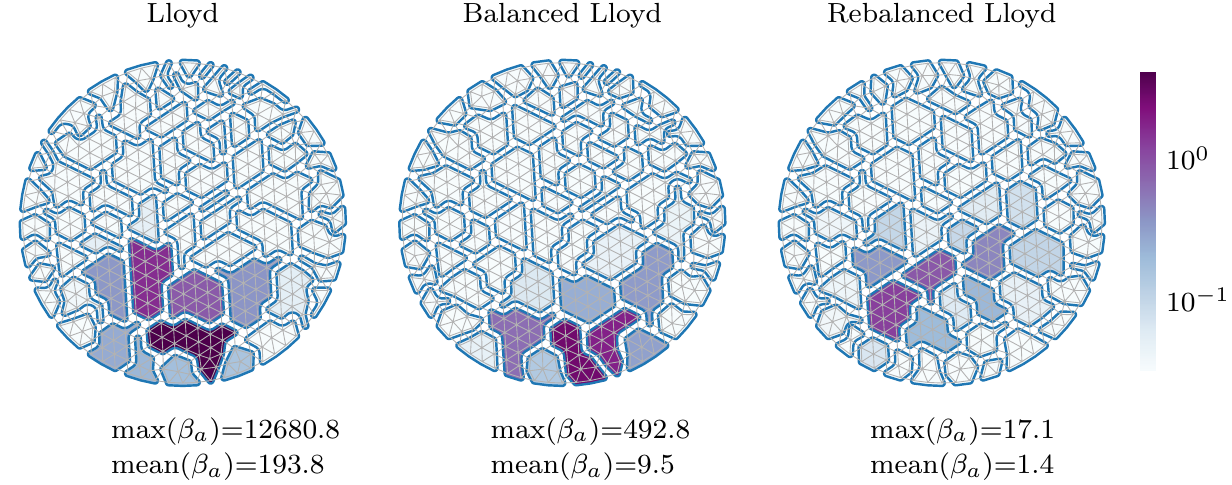}
  \caption{Localizing $\beta$ to each cluster via \cref{eq:betaopt2}.}\label{fig:disc_compareenergy_aggregates}
\end{figure}
\begin{figure}
  \centering
  \includegraphics[width=0.5\textwidth]{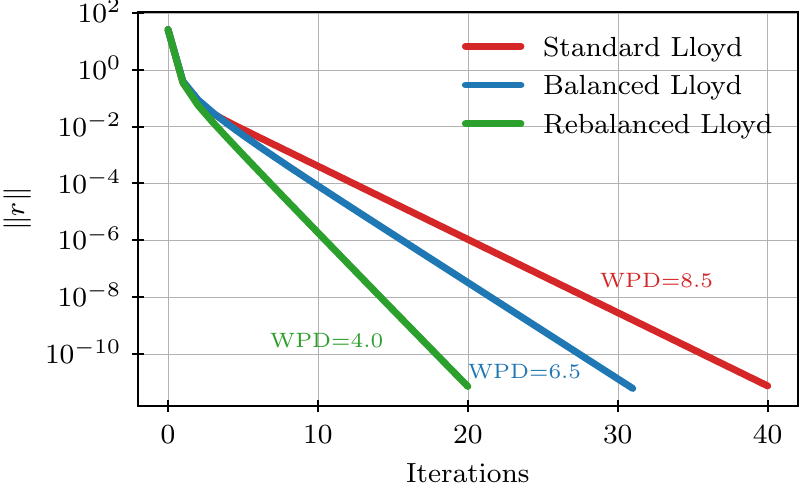}
  \caption{Example convergence for two-level AMG with different clusterings.}\label{fig:disc_compareenergy_convergence}
\end{figure}

\subsection{Additional problems in Algebraic Multigrid}

As additional evidence of the effectiveness of rebalanced Lloyd clustering, we
consider several examples in both 2D and 3D.

{\bf 3D restricted channel:}
    The 3D domain $\Omega$ is defined by a spline on the points
    \begin{equation*}
    \left[(0, 4, -8),
           (0, 4, -6),
           (0, 1,  0),
           (0, 4,  6),
           (0, 4,  8)\right],
    \end{equation*}
    rotated about the $z$-axis, (see~\cref{tab:additional_examples}).
    A 3D tetrahedral mesh with \num{16921} elements is generated with
    Gmsh~\cite{gmsh} through pygmsh~\cite{Schlomer_pygmsh_A_Python}.
We use Firedrake~\cite{Rathgeber2016} to
discretize~\eqref{eq:modelproblem} with linear finite elements on
tetrahedra, and select an average of 25 points per cluster.

{\bf 2D restricted channel:} The 2D domain $\Omega$ is defined by $[-2,
    2]\times[-1, 1]\setminus C$ with $C = C^+ \cup C^-$, for $C^\pm$
    representing discs of radius 0.8 at $(0, \pm 1)$
    (see~\cref{tab:additional_examples}).  As for the 3D restricted channel, we use
    Gmsh to generate a graded, triangular mesh with \num{5832} elements, with a
    characteristic length of 0.012 at the center and growing to 0.12 at the
    left/right edges. This forces tighter clustering toward the center, as
    shown in~\cref{tab:additional_examples}. The discretization matrix
    for~\eqref{eq:modelproblem} is
    constructed with linear finite elements, and we target clusters of
    size 8.

{\bf 2D anisotropic diffusion:} The 2D domain is defined by the unit
square, and we consider the problem $-\nabla\cdot K \nabla u = f$ with
pure Dirchlet conditions.  We define the anisotropic diffusion tensor
as
$
K =
  \begin{bmatrix}
    \cos\theta & -\sin\theta\\
    \sin\theta &  \cos\theta
  \end{bmatrix}
  \begin{bmatrix}
    1 & \\
      & \varepsilon
  \end{bmatrix}
  \begin{bmatrix}
    \cos\theta & -\sin\theta\\
    \sin\theta &  \cos\theta
  \end{bmatrix}^T
$,
for $\varepsilon = 0.1$ and $\theta = \pi/3$.  We discretize this on a
$42\times 42$ uniform mesh  (with \num{1681} elements) and
Q1 bilinear elements, and specify a target cluster size of 12.

{\bf P2 elements:}  The 2D domain is a unit disc, on which we consider~\eqref{eq:modelproblem}.  A triangular mesh is constructed with \num{982}
elements and P2 quadratic finite elements are used to generate the discretization
matrix.  We specify 5 nodes per cluster.

In each of the examples of~\cref{tab:additional_examples}, a zero right-hand
side is used to assess convergence of the smoothed aggregation multigrid
solver.  From the convergence histories, we see that rebalanced Lloyd
clustering improves solver convergence, even for these relatively benign
problems.  For the restricted channel problems, the resulting clustering
resembles the expected isotropic behavior with well rounded clusters.
Likewise, in the case of anisotropy, we see that the clustering mimics the
diffusion direction, while maintaining balance across clusters. Finally, the
P2 case reveals the benefit of specifying the coarsening ratio: in this
case, the coarsening ratio of 1/5 outperforms greedy coarsening (which
yields a ratio of around
1/10).

As a final example, \cref{tab:additional_examples} highlights a parallel
partitioning of an arc heated combustion channel at the University of Illinois Urbana-Champaign\footnote{
\url{https://tonghun.mechse.illinois.edu/research/hypersonics-act-ii/}, \url{https://ceesd.illinois.edu/}
}. In this case, rebalanced Lloyd effectively partitions the ${\sim}100$k mesh elements, keeping refined features such as the injector local to a cluster.

\begin{table}
  \centering
\setlength{\aboverulesep}{-0pt}
\setlength{\belowrulesep}{-0pt}

\begin{tabular}{
     m{0.26\textwidth} @{\hspace{0.1\tabcolsep}}
     m{0.4\textwidth} @{\hspace{0.1\tabcolsep}}
     m{0.3\textwidth}
  }
  \toprule
  \parbox{1.8cm}{3D\\ restricted channel} &
  \includegraphics[width=\dimexpr\linewidth-2\fboxsep-2\fboxrule\relax]
    {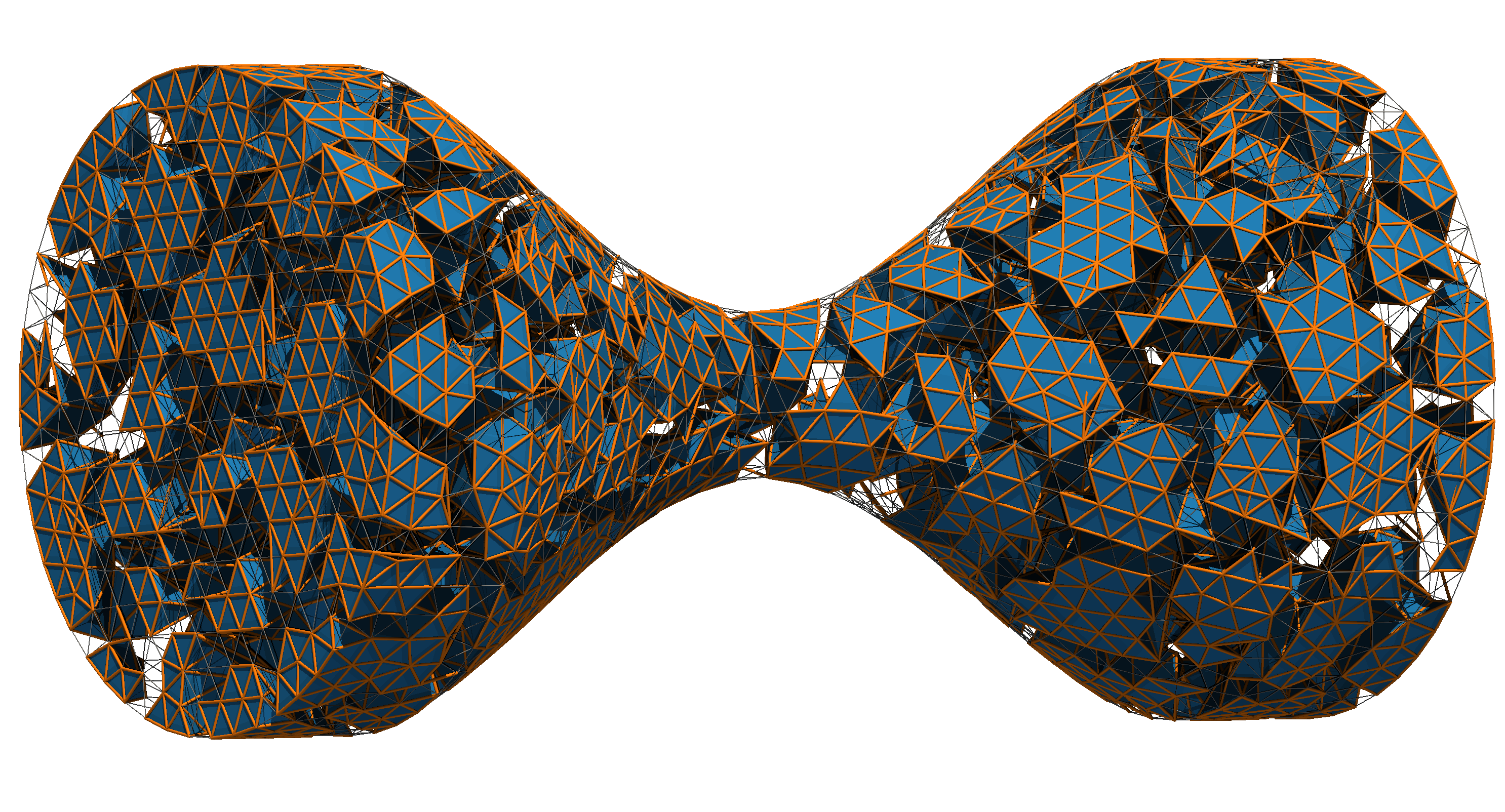}
                        &
  \includegraphics[width=\dimexpr\linewidth-2\fboxsep-2\fboxrule\relax]
   {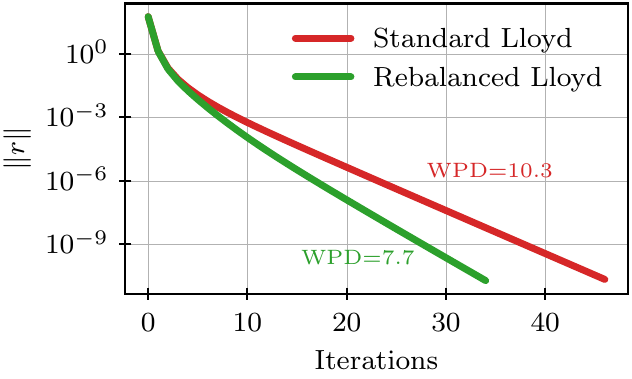}
  \\
  \midrule
  \parbox{1.8cm}{2D\\ restricted channel} &
  \includegraphics[width=\dimexpr\linewidth-2\fboxsep-2\fboxrule\relax]
   {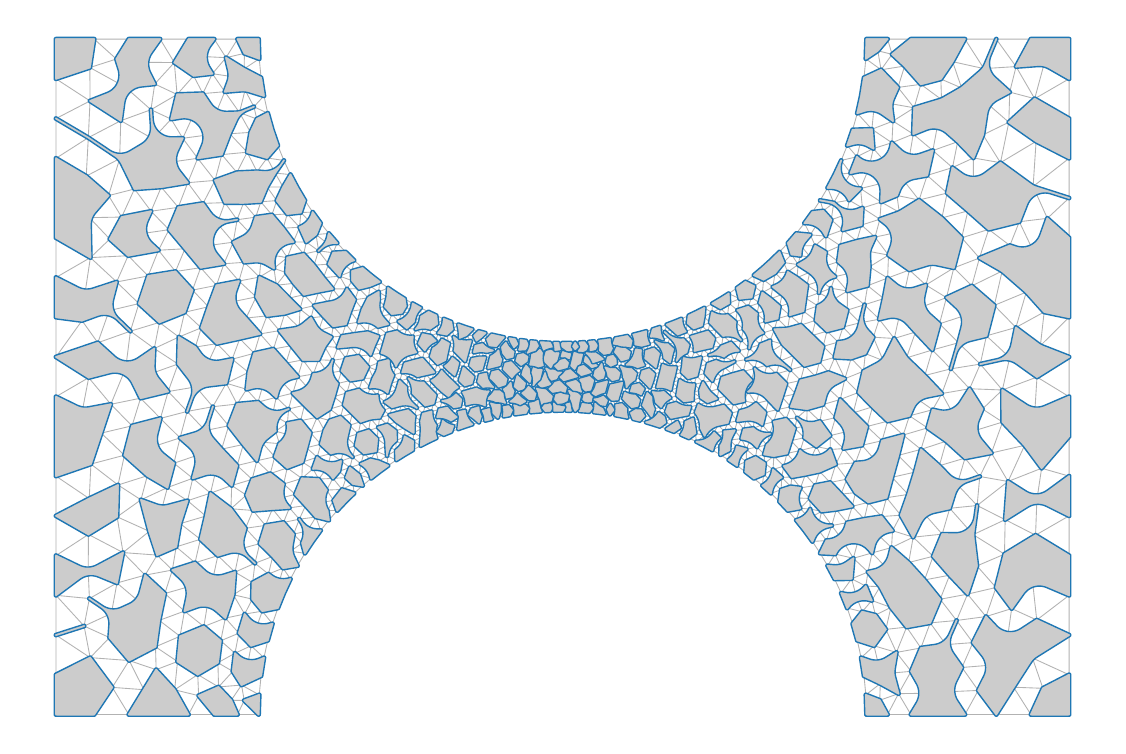}
  &
  \includegraphics[width=\dimexpr\linewidth-2\fboxsep-2\fboxrule\relax]
   {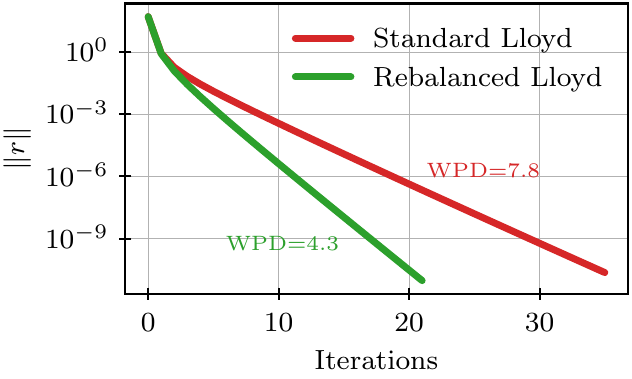}
  \\
  \midrule
  \parbox{1.8cm}{2D\\ anisotropic diffusion} &
  \includegraphics[width=\dimexpr\linewidth-2\fboxsep-2\fboxrule\relax]
   {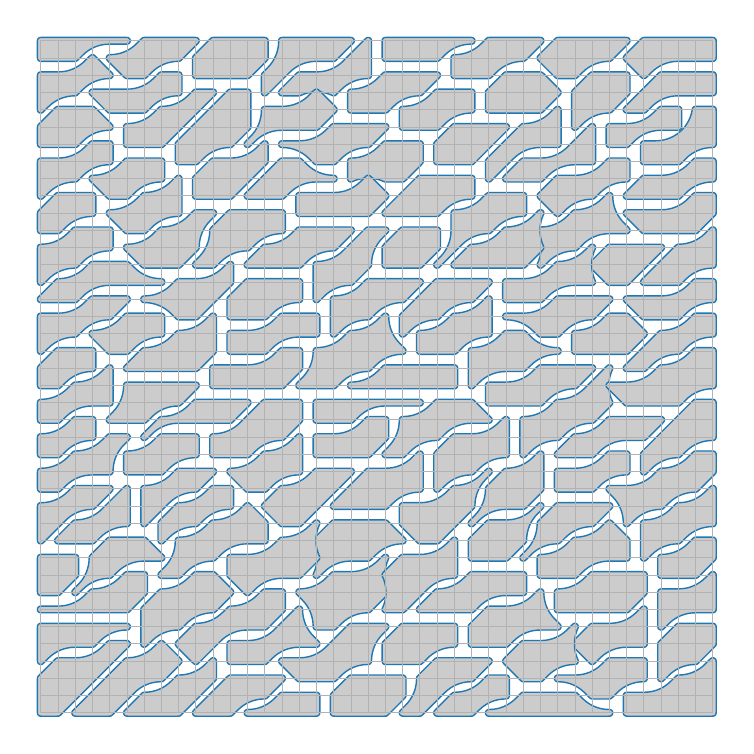}
  &
  \includegraphics[width=\dimexpr\linewidth-2\fboxsep-2\fboxrule\relax]
   {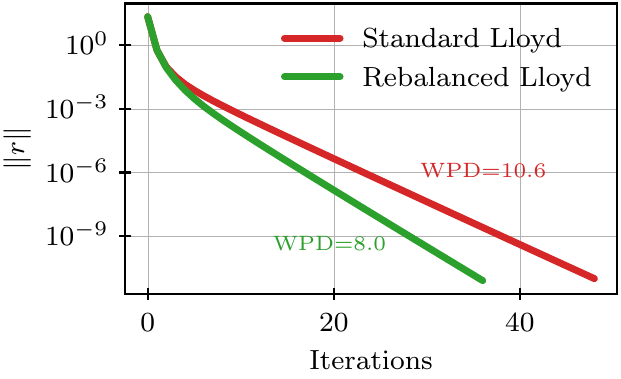}
  \\
  \midrule
  \parbox{1.8cm}{P2\\ elements} &
  \includegraphics[width=\dimexpr\linewidth-2\fboxsep-2\fboxrule\relax]
   {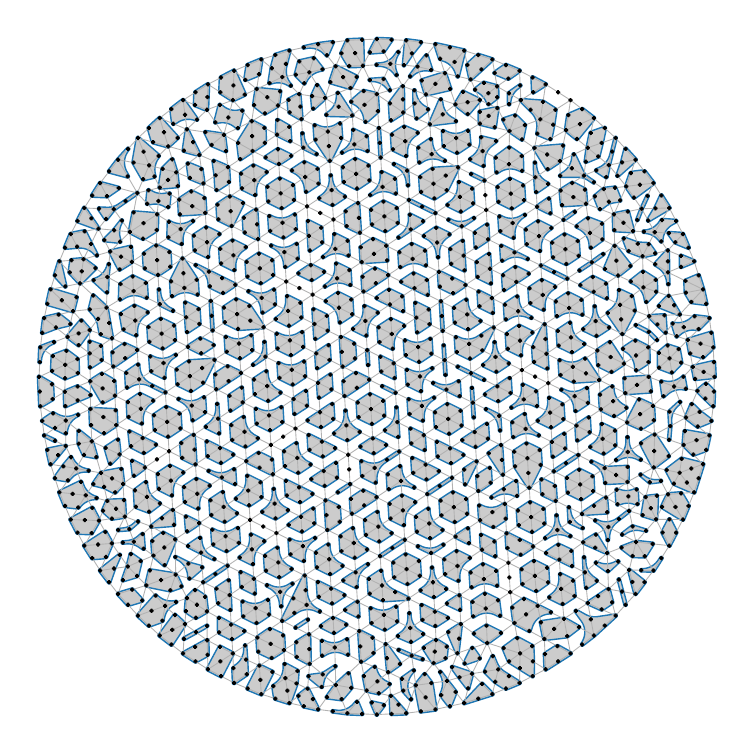}
  &
  \includegraphics[width=\dimexpr\linewidth-2\fboxsep-2\fboxrule\relax]
   {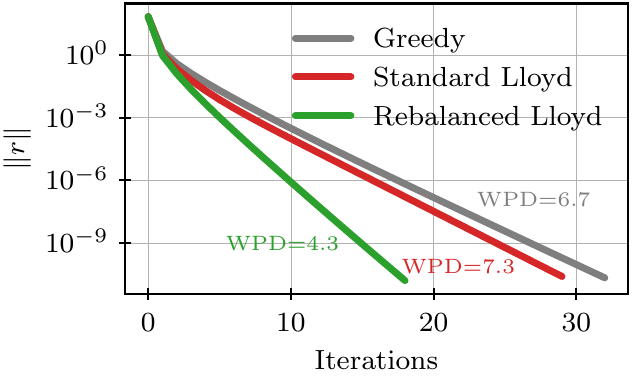}
  \\
  \midrule
  \multicolumn{3}{c}{
  ACTII mesh
  }\\
  \multicolumn{3}{c}{
    \includegraphics[width=0.95\linewidth]{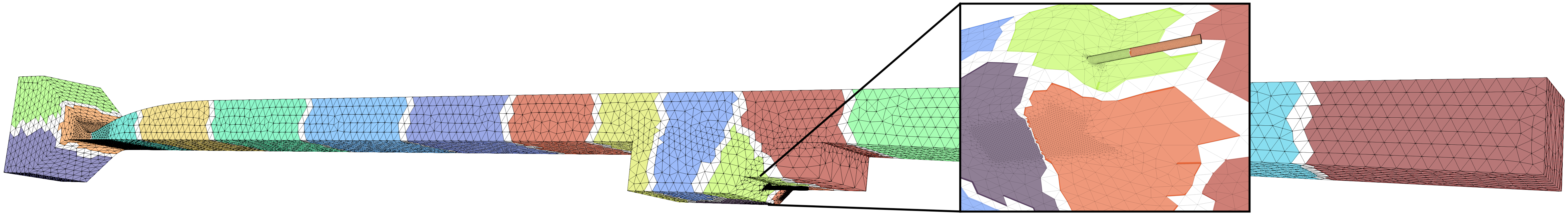}
  }\\
  \bottomrule
\end{tabular}
\caption{Additional examples.
  \textit{ACTII mesh credit}: Mike Anderson at UIUC.
}\label{tab:additional_examples}
\end{table}

\section{Conclusions and extensions}

In this paper, we study and extend the use of Lloyd's algorithm for determining
clusters in graphs.  Our proposed \textit{balanced} and \textit{rebalanced}
Lloyd clustering algorithms are linear in time, guarantee connected clusters,
and are consistent with minimizing a quadratic energy functional.
In addition, the algorithms are implemented in Python/C++ and are available
through the open source project PyAMG~\cite{BeOlSc2022}.
One major
topic for future work is the choice of that energy functional; while the steps
in the algorithms above are consistent with an $\ell^2$-distance style energy,
they can easily be extended to other energy functionals in a consistent way.
Theoretical guidance is clearly needed to determine the proper choice of such a
functional.  We also note that we consider only serial algorithms in this
paper; properly extending these approaches to their parallel counterparts is
also an important subject for future research.

\bibliographystyle{siamplain}
\bibliography{refs-lloyd}

\appendix

\section{Review of algebraic multigrid methods}\label{app:AMG}
Algebraic multigrid methods seek to approximate solutions to sparse linear
systems of the form
\begin{equation}
  A u = f
\end{equation}
for $A\in\mathbb{R}^{\Nnode\times\Nnode}$, and $u,f\in\mathbb{R}^{\Nnode}$. Here, we outline \textit{aggregation}-based AMG
methods for use as an application in the development of Lloyd-style clustering.
The set of indices, $\{1,\dots,\Nnode\}$, enumerate the degrees of freedom
(DoFs) and represent the fine level in the multilevel grid hierarchy. This set
is partitioned and grouped into disjoint clusters, see \cref{def:clustering}.

Each cluster represents a node in the
coarse grid and, collectively, the cluster mapping defines a tentative restriction
operator, $\hat{R}$, as
\begin{equation}\label{eq:restriction}
  \hat{R}_{a,i}\
  = \begin{cases}
    1 & \text{if vertex $i$ is in cluster $a$,} \\
    0 & \text{otherwise.}
  \end{cases}
\end{equation}
An example with 12 fine nodes and 3 coarse nodes (clusters) is given in~\cref{fig:disc_agg_R}; the pattern for (the transpose of) $\hat{R}$ is also illustrated.
\begin{figure}
  \centering
  \includegraphics[height=1.35in]{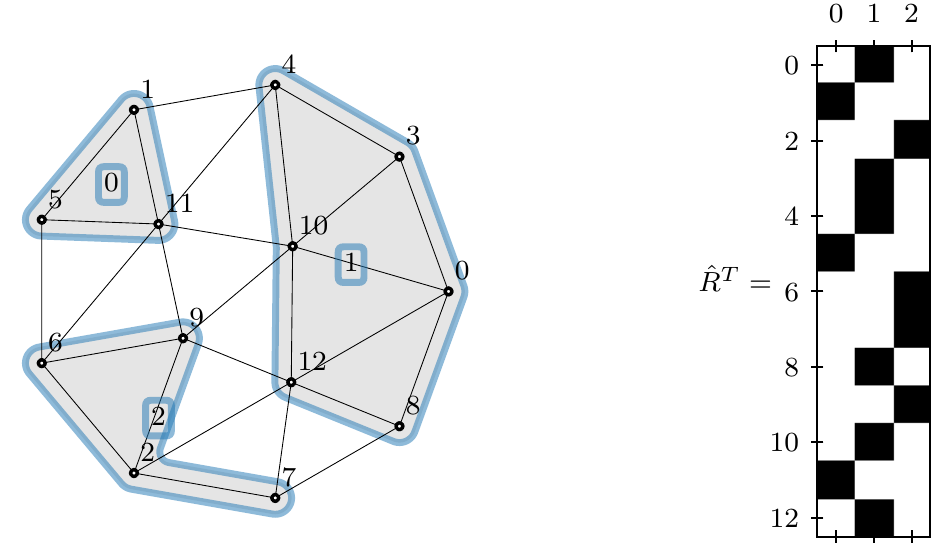}
  \caption{Example clustering and restriction matrix.}\label{fig:disc_agg_R}
\end{figure}

The restriction pattern defines the tentative interpolation pattern through $\hat{Z} = \hat{R}^T$.
Smoothed aggregation (SA) AMG proceeds by using the nonzero pattern of $\hat{Z}$ as a partition of unity to localize a given global set of vectors, $C$, defining the near-null space of matrix $A$ and, then, smoothing each column of the resulting matrix, $Z$, with (for example) weighted Jacobi.  This defines the smoothed interpolation operator, $Z$, from which
a coarse-level operator is defined over cluster DoFs as
$A_c = Z^T A Z$.

The complete algorithm for constructing SA AMG is given in~\cref{alg:sa}, where
we note the omission of several details (and optional parameters, denoted by
$[\textit{opt}]$) since the focus of this work is primarily
on~\cref{line:saaggregate}.  We refer the reader to~\cite{vanaek1996algebraic, brezina2001convergence, KStuben_2001a} for a
more complete description and analysis of aggregation-based AMG methods.  Here, we note
that \cref{line:saedgeweights} is critically important to the convergence of
the method; in practice, unit weights or algebraic distances \textit{can} be
used, yet generalized measures such as the \textit{evolution}
measure~\cite{2010_OlScTu_evosoc} (used in~\cref{sec:numerics}) have proven robust in practice.
\begin{algorithm}[!ht]
\caption{Smoothed aggregation~---~setup}\label{alg:sa}
\begin{algorithmic}[1]
\Function{sa-setup}{$A_0$, $\Nlevel$, $C$}
\For {$\ell \gets 0,\ldots,\Nlevel-1$}
\State $W \gets \textsc{edge-weights}(A_{\ell}, [\textit{opt}])$\Comment{determine strong edges in graph of $A$}\label{line:saedgeweights}
\State $m, c \gets \textsc{cluster}(W, [\textit{opt}])$\Comment{cluster membership and centers}\label{line:saaggregate}
\State $Z_{\ell} \gets \textsc{interpolation}(m, C, [\textit{opt}])$\Comment{form interpolation}
\State $A_{\ell+1} = Z_{\ell}^T A_{\ell} Z_{\ell}$\Comment{construct coarse-level operator}
\EndFor
\State \textbf{return} $\{A_\ell\}_0^{\Nlevel}$, $\{Z_\ell\}_0^{\Nlevel-1}$
\EndFunction
\end{algorithmic}
\end{algorithm}

With a multigrid hierarchy of coarse operators and interpolation, multigrid (MG) iterates
via the familiar V-cycle as in~\cref{alg:mgcycle}.  In \cref{sec:numerics}, we have considered both two-level ($\Nlevel=2$) and multilevel results, underscoring improved convergence
by improving the clustering, while leaving the other multigrid parameters untouched.  As we use a subscript within these algorithms to denote the level within the multigrid hierarchy, we use a superscript to indicate the multigrid iteration number, with $u^{(k+1)} = $\textsc{MG-V-CYCLE}($A_0,\dots,A_{\Nlevel}$,$Z_0,\dots,Z_{\Nlevel-1}$, $u^{(k)}$, $f$).
\begin{algorithm}[!ht]
\caption{MG cycle}\label{alg:mgcycle}
\begin{algorithmic}[1]
\Function{mg-v-cycle}{$A_0,\dots,A_{\Nlevel}$,$Z_0,\dots,Z_{\Nlevel-1}$, $u_0$, $f_0$}
\For {$\ell=0,\dots,\Nlevel-1$}
  \State $u_{\ell} \gets \textsc{relax}(A_{\ell}, u_{\ell}, f_{\ell})$\Comment{fixed number of relaxation sweeps}
  \State $f_{\ell+1} \gets Z_{\ell}^T (f_{\ell} - A_{\ell} u_{\ell})$\Comment{compute restricted residual}
\EndFor
\State $u_{\Nlevel} \gets A_{\Nlevel}^{-1} f_{\Nlevel}$\Comment{solve coarsest level problem}
\For {$\ell=\Nlevel-1,\dots,0$}
  \State $u_{\ell} \gets u_{\ell} + Z_{\ell} u_{\ell}$\Comment{interpolate and correct}
  \State $u_{\ell} \gets \textsc{relax}(A_{\ell}, u_{\ell}, f_{\ell})$\Comment{fixed number of relaxation sweeps}
\EndFor
\State \textbf{return} $u_0$
\EndFunction
\end{algorithmic}
\end{algorithm}

\section{Standard AMG clustering algorithms}\label{sec:standard-agg}
\subsection{Greedy clustering}\label{sec:greedy-agg}

Greedy clustering (also known as ``greedy aggregation'' or ``standard aggregation'') was first introduced by M{\'{\i}}ka and
Van{\v{e}}k~\cite{mika1992acceleration}; we
use a close variant.
Greedy clustering consists of two passes over the set of nodes of the graph. In the first pass, for each node, if
all neighbors in the graph remain unclustered, then the node becomes a center, forming a cluster from the node and its neighborhood. In the second pass, each
unclustered node is included in a neighboring cluster, if possible. If a
neighboring cluster is not found, then the unclustered node is considered a
center node and the node with its unclustered neighbors form a new cluster.
In the case of multiple neighboring clusters, there are several options:
arbitrary selection, index, size, or magnitude of the weight can each be used
to determine cluster membership. The full greedy algorithm is given in
\cref{alg:greedy-agg}.
\begin{algorithm}[!ht]
\caption{Greedy clustering. See \cref{tab:symbols} for variable definitions.}\label{alg:greedy-agg}
\begin{algorithmic}[1]
\Function{greedy-clustering}{$W$}
\State $m_i \gets 0$ for all $i = 1,\ldots,\Nnode$\Comment{initially all nodes are unclustered}
\State $a \gets 1$\Comment{first cluster index}
\For {$i \gets 1,\ldots,\Nnode$}\Comment{\textbf{first pass}}
  \If {$m_i=0$ and $m_j=0$ for all $j$ s.t.\ $W_{i,j}\ne 0$}\Comment{unclustered}
    \State $m_i \gets a$\Comment{add $i$ and neighbors to cluster $a$}
    \State $m_j \gets a$, for all $j$ s.t.\ $W_{i,j}\ne 0$
    \State $c_a \gets i $\Comment{mark cluster center}
    \State $a \gets a + 1$\Comment{increment cluster index}
  \EndIf
\EndFor
\For {$i \gets 1,\ldots,\Nnode$}\Comment{\textbf{second pass}}
  \If {$m_i = 0$}\Comment{unclustered}
    \If {$\exists\ j$ s.t.\ $W_{i,j}\ne 0$ and $m_j> 0$}\Comment{clustered neighbor}
      \State $j \gets \argmax\limits_{j \,:\, m_j > 0} W_{i,j}$\Comment{neighbor with largest weight}
      \State $m_i \gets m_j$
    \Else\Comment{form new cluster}
      \State $m_i \gets a$
      \For {$j$ such that $W_{i,j}\ne0$ and $m_j=0$}
        \State $m_j \gets a$
      \EndFor
      \State $a \gets a + 1$\Comment{increment cluster index}
    \EndIf
  \EndIf
\EndFor
\State \textbf{return} $m, c$
\EndFunction
\end{algorithmic}
\end{algorithm}

\subsection{Maximal independent set based clustering}\label{sec:mis}

The greedy algorithm is inherently serial, yet there are two
immediate observations.
First, any two center nodes of two (distinct) clusters
must be more than two edges apart.
Second, if an unclustered node is more than two
edges from any existing center, then the node is eligible
to be a center of a new cluster. Hence, the center nodes
from the greedy algorithm represent a distance-2
maximal independent set or MIS(2).  This leads to the MIS(2) clustering algorithm, where
an MIS(2) over the nodes is first constructed, followed by construction of the  clustering using the MIS(2) center nodes.
This has been shown to exhibit a high degree of parallelism~\cite{bell2012exposing}; see~\cite[Algorithm 5]{bell2012exposing} for details.

Given a distance-2 maximal independent set, the clustering
process is straightforward. In the first step, the index of
the cluster representing the center is propagated to its neighbors.
This continues in the second step, where the index of the cluster
is propagated to the second layer of neighbors;
if there are multiple clusters adjacent to an unclustered
node, the choice is made arbitrarily (or by index). The algorithm is shown
in \cref{alg:mis2-agg}.
\begin{algorithm}[!ht]
\caption{MIS(2) clustering. See \cref{tab:symbols} for variable definitions.}\label{alg:mis2-agg}
\begin{algorithmic}[1]
\Function{mis(2)-clustering}{$W$}
\State $c \gets \textsc{mis(\(W\), 2)}$\Comment{distance-2 independent set}
\State $m_i \gets 0$ for $i=1,\dots,\Nnode$
\State $\Ncluster \gets |c|$
\For {$a=1,\dots,\Ncluster$}\Comment{\textbf{pass 1:} distance-1}
  \State $i \gets c_a$\Comment{index of center for cluster $a$}
  \State $m_i \gets a$\Comment{set cluster number for center}
  \For {$j$ s.t.\ $W_{i,j}\ne 0$}
    \State $m_j \gets a$\Comment{set cluster number for neighbors}
  \EndFor
\EndFor
\For {$i$ s.t.\ $m_i > 0$}\Comment{\textbf{pass 2:} distance-2}
  \For {$j$ s.t.\ $W_{i,j}\ne 0$ and $m_j = 0$}
    \State $m_j \gets m_i$\Comment{set cluster number for neighbors}
  \EndFor
\EndFor
\State \textbf{return} $m, c$
\EndFunction
\end{algorithmic}
\end{algorithm}

With an appropriate ordering, the first pass of MIS-based and greedy
clustering can yield identical clusters.  With only minor differences in
the second pass, we expect the clustering patterns to be similar.  Indeed, the
convergence factors of AMG based on these two clustering strategies are shown
to be close in practice~\cite[Appendix]{bell2012exposing}.

\section{Notation}\label{app:notation}
\Cref{tab:symbols} summarizes the notation.
\begin{table}[b]
  \centering
  \begin{tabular}{lp{81mm}l}
    \toprule
    Symbol & Definition & Domain \\
    \midrule
    $A$ & left hand side operator in the linear system $A u = f$ & $\mathbb{R}^{N_{\rm node} \times N_{\rm node}}$ \\
    $a$ & cluster index & $\{1,\ldots,N_{\rm cluster}\}$ \\
    $B$ & set of border nodes between clusters; $B \subseteq V$ & $\mathcal{P}(V)$ \\
    $c_a$ & center node index for cluster $a$ & $V$ \\
    $c^1_a,c^2_a$ & new cluster centers if cluster $a$ is split; see \cref{alg:split_improvement} & $V$ \\
    $\delta$ & second-term energy scaling coefficient, see \cref{eq:delta} & $\mathbb{R}$ \\
    $d_i$ & shortest-path distance to node $i$ from nearest center; $d_i = \infty$ if node $i$ is not in a cluster & $\overbar{\mathbb{R}}_{\ge 0}$ \\
    $D_{i,j}$ & shortest-path distance from node $i$ to $j$ within a single cluster; $D_{i,j} = \infty$ if there is no such path $i \to j$ & $\overbar{\mathbb{R}}_{\ge 0}$ \\
    $\Delta_{\rm min}$ & minimum difference between distinct values of $W_{i,j}$ & $\mathbb{R}$ \\
    $E$ & set of edges in the graph $G$ & $\mathcal{P}(V \times V)$ \\
    $f$ & right hand side of the linear system $A u = f$ & $\mathbb{R}^{N_{\rm node}}$ \\
    $G$ & graph with nodes $V$, edges $E$, and weights $W$ & \\
    $H$ & shortest-path energy function, see \cref{eq:energy} & $\mathbb{R}$ \\
    $H_\delta$ & energy function minimized by clustering, see \cref{eq:squared_energy} & $\mathbb{R}$ \\
    $i,j,k$ & node indices & $V$ \\
    $L_a$ & energy increase if cluster $a$ is eliminated; see \cref{alg:elimination_penalty} & $\mathbb{R}$ \\
    $M_a$ & whether cluster $a$ is modifiable during rebalancing & $\{\text{True},\text{False}\}$ \\
    $m_i$ & cluster index (membership) containing node $i$ & $\{1,\dots,\Ncluster\}$ \\
    $\Nnode$ & number of nodes & $\mathbb{N}_1$ \\
    $\Ncluster$ & number of clusters & $\mathbb{N}_1$ \\
    $n_i$ & number of nodes with $i$ as predecessor & $\mathbb{N}_0$ \\
    $P_{i,j}$ & predecessor index for node $j$ on the shortest path $i \to j$ within a cluster; $P_{i,j} = 0$ if there is no path $i \to j$ & $V \cup \{0\}$ \\
    $p_i$ & predecessor index for node $i$ on the shortest path from its cluster center; $p_i = 0$ if node $i$ is not in a cluster & $V \cup \{0\}$ \\
    $q_i$ & sum of squared distances from node $i$ to all other nodes in the same cluster & $\mathbb{R}_{\ge 0}$ \\
    $S_a$ & energy decrease if cluster $a$ is split; see \cref{alg:split_improvement} & $\mathbb{R}$ \\
    $s_a$ & size (number of nodes) of cluster $a$ & $\mathbb{N}_1$ \\
    $\mathfrak{s}_i$ & size (number of nodes) of the cluster containing node $i$; $\mathfrak{s}_i = 0$ if node $i$ is not in a cluster & $\mathbb{N}_0$ \\
    $T$ & total number of time/iteration steps taken by an algorithm
    ($T_{\text{max}}$ and $T_{\text{BFmax}}$ denote the maximum)
        & $\mathbb{N}_0$ \\
    $t$ & time/iteration index & $\mathbb{N}_0$ \\
    $u$ & solution vector in the linear system $A u = f$ & $\mathbb{R}^{N_{\rm node}}$ \\
    $V$ & set of nodes in the graph; $V = \{1,\ldots,N_{\rm node}\}$ & $\mathcal{P}(\mathbb{N}_1)$ \\
    $V_a$ & set of nodes in cluster $a$; $V_a \subseteq V$ & $\mathcal{P}(V)$ \\
    $W_{i,j}$ & weighted adjacency matrix of the graph where $W_{i,j}$ is the edge weight $i \to j$ & $\mathbb{R}$ \\
    $z$ & generic variable placeholder & --- \\
    \bottomrule
  \end{tabular}
  \caption{List of symbols. Here $\mathcal{P}()$ denotes the power set and $\overbar{\mathbb{R}}$ is the extended reals.}\label{tab:symbols}
\end{table}

\end{document}